\documentclass[11pt]{article}
\usepackage{amssymb,oldlfont}
\usepackage{amsmath}
\usepackage{hyperref}
\usepackage{color}
\definecolor{myorange}{RGB}{180,90,0}
\definecolor{mygreen}{RGB}{70,140,0}
\usepackage[normalem]{ulem}

\usepackage[T1]{fontenc}
\usepackage[utf8]{inputenc}
\def\wrtext#1{\relax\ifmmode{\leavevmode\hbox{#1}}\else{#1}\fi}
\def\abs#1{\left|#1\right|}
\def\begeq{\begin{equation}}
\def\endeq{\end{equation}}

\def\part#1{\frac{\partial}{\partial #1}}

\def\norm#1{||\,#1\,||}

\newcommand{\real}{\mbox{\bf R}}
\newcommand{\comp}{\mbox{\bf C}}

\newcommand{\nat}{\mbox{\bf N}}

\renewcommand{\Re}{\mbox{\rm Re\,}}

\def\Re{{\rm Re\,}}

\newtheorem{dref}{Definition}[section]

\newtheorem{theo}[dref]{Theorem}
\newtheorem{prop}[dref]{Proposition}

\newenvironment{proof}{\vspace{.3cm}\noindent{{\em Proof:}}}{\hfill$\Box$}

\title{Asymptotics for Bergman projections with smooth weights: a direct approach}
\author{Michael \textsc{Hitrik} \footnote{Department of Mathematics, UCLA, Los Angeles CA 90095-1555, USA, {\sf hitrik@math.ucla.edu}} \and Matthew \textsc{Stone} \footnote{Department of Mathematics, UCLA, Los Angeles CA 90095-1555, USA, {\sf stoney31415@math.ucla.edu}}}
\date{}

\begin{document}
\maketitle

\vspace*{1cm}
\noindent
{\bf Abstract}: We adapt the direct approach to the semiclassical Bergman kernel asymptotics, developed recently in~\cite{DelHiSj} for real analytic exponential weights, to the smooth case. Similar to~\cite{DelHiSj}, our approach avoids the use of the Kuranishi trick and it allows us to construct the amplitude of the asymptotic Bergman projection by means of an asymptotic inversion of an explicit Fourier integral operator.

\section{Introduction}
\setcounter{equation}{0}
\label{sec_introduction}
Let $\Omega \subset \comp^n$ be an open pseudoconvex domain and let $\Phi\in C^{\infty}(\Omega)$ be a strictly plurisubharmonic function. Exponentially weighted spaces of holomorphic functions of the form $H_{\Phi}(\Omega) = {\rm Hol}(\Omega) \cap L^2(\Omega, e^{-2\Phi/h})$ occur naturally in analytic microlocal analysis~\cite{Sj82},~\cite{GrSjDuke},~\cite{Sj96},~\cite{MelSj},~\cite{HiSj15}, among other areas, when passing from the real to the complex domain by means of an FBI transform~\cite{Sj82},~\cite{Sj96}. Associated to the space $H_{\Phi}(\Omega)$ is the orthogonal (Bergman) projection
\begeq
\label{eq1.1}
\Pi: L^2(\Omega, e^{-2\Phi/h}) \rightarrow H_{\Phi}(\Omega),
\endeq
and complex microlocal techniques have long been known to be useful in the study of the asymptotic behavior of $\Pi$ in the semiclassical limit $h\rightarrow 0^+$. The existence of a complete asymptotic expansion for the Schwartz kernel of $\Pi$ close to the diagonal has been established in the pioneering contributions~\cite{Catlin},~\cite{Z98}, in the context of high powers of a holomorphic line bundle with positive curvature, over a complex compact manifold. See also~\cite{Charles},~\cite{DLM},~\cite{MM}, as well as~\cite{E1},~\cite{E2} for the case of domains in $\comp^n$. The original proofs in~\cite{Catlin},~\cite{Z98} relied on the description of singularities of the Szeg\H{o} kernel on the boundary of a strictly pseudoconvex smooth domain given by the Boutet de Monvel -- Sj\"ostrand parametrix in the seminal work~\cite{BoSj}, which in turn depended, in particular, on the theory of Fourier integral operators with complex phase functions developed in~\cite{MelSj74}. More self-contained explicit approaches to the Bergman kernel asymptotics have subsequently been developed in~\cite[Section 3]{MelSj} and~\cite{BBSj}, with the former work starting with the approximate projection property $\widetilde{\Pi}^2 = \widetilde{\Pi} + {\cal O}(h^{\infty})$, while the approach of~\cite{BBSj} focuses more directly on the reproducing property of the Bergman projection on the space $H_{\Phi}(\Omega)$.

\medskip
\noindent
To motivate a bit further, let us recall that the basic idea of the approach of~\cite{BBSj} consists, roughly speaking, of expressing the identity operator on $H_{\Phi}(\Omega)$ in such a way that it automatically becomes the asymptotic Bergman projection, at least locally. This is accomplished, essentially, by first representing the identity as an $h$--pseudodifferential operator in the complex domain, and then passing to a non-standard phase via a suitable change of variables, usually referred to as the Kuranishi trick. See~\cite[Chapter 4]{Sj82},~\cite[Section 2.2]{BBSj}, and also~\cite[Chapter 3]{GrSj} for the standard application of the Kuranishi trick to changes of variables for pseudodifferential operators in the real domain. Now the Kuranishi trick becomes somewhat complicated to execute in situations when the Levi form $i\partial \overline{\partial}\Phi\geq 0$ of the weight $\Phi$ becomes almost degenerate in some directions. Such nearly degenerate weights occur naturally, in particular, in the work in progress~\cite{HiSj}, devoted to a heat evolution approach to second microlocalization with respect to a real analytic hypersurface. A direct approach to the semiclassical asymptotics for Bergman projections, not relying upon any changes of variables, has recently been developed in~\cite{DelHiSj} in the non-degenerate case, assuming that the weight $\Phi$ is real analytic. The approach of~\cite{DelHiSj} has allowed, in particular, to give a quick proof of a result of~\cite{Del},~\cite{RSV}, stating that in the analytic case, the amplitude of the asymptotic Bergman projection is given by a classical analytic symbol. Our purpose here is to adapt the approach of~\cite{DelHiSj} to the $C^{\infty}$ case. The following is the main result of this work.

\begin{theo}
\label{Theorem1}
Let $\Omega \subset \comp^n$ be open and let $\Phi \in C^{\infty}(\Omega;\real)$ be strictly plurisubharmonic in $\Omega$. Let $x_0\in \Omega$. There exist a classical elliptic symbol $a(x,\widetilde{y};h)\in S^0_{{\rm cl}}({\rm neigh}((x_0,\overline{x_0}),\comp^{2n}))$ of the form
$$
a(x,\widetilde{y};h) \sim \sum_{j=0}^{\infty} h^j a_j(x,\widetilde{y}),
$$
in $C^{\infty}$, with $a_j \in C^{\infty}({\rm neigh}((x_0,\overline{x_0}),\comp^{2n}))$, holomorphic to $\infty$--order along the anti-diagonal $\widetilde{y} = \overline{x}$, satisfying
\begeq
\label{eq1.2}
(A a)(x,\overline{x};h) = 1 + {\cal O}(h^{\infty}), \quad x\in {\rm neigh}(x_0,\comp^n),
\endeq
where $A$ is an elliptic Fourier integral operator, and small open neighborhoods $U \Subset V \Subset \Omega$ of $x_0$, with $C^{\infty}$--boundaries, such that the operator
\begeq
\label{eq1.3}
\widetilde{\Pi}_{V} u(x) = \frac{1}{h^n} \int_{V} e^{\frac{2}{h}\Psi(x,\overline{y})} a(x,\overline{y};h)  u(y) e^{-\frac{2}{h}\Phi(y)}\, L(dy)
\endeq
satisfies
\begeq
\label{eq1.4}
\widetilde{\Pi}_V - 1 = {\cal O}(h^{\infty}): H_{\Phi}(V) \rightarrow L^2(U, e^{-2\Phi/h}\, L(dx)).
\endeq
Here in {\rm (\ref{eq1.3})}, the $C^{\infty}$ function $\Psi$ is holomorphic to $\infty$--order on the anti-diagonal, $\Psi(x,\overline{x}) = \Phi(x)$, and $L(dy)$ is the Lebesgue measure on $\comp^n$.
\end{theo}

\medskip
\noindent
When proving Theorem \ref{Theorem1}, we proceed largely along the lines of~\cite{DelHiSj}, while making essential use of the techniques of almost holomorphic extensions~\cite{Ho69},~\cite{Ma},\cite{MelSj74}. Compared to~\cite{DelHiSj}, we also have to rely on the $L^2$ estimates for the $\overline{\partial}$ operator~\cite{Ho65} even more, to account for the fact that functions in the range of the operator $\widetilde{\Pi}_V$ in (\ref{eq1.3}) are not quite holomorphic. The plan of the paper is as follows: In Section \ref{sec:asymp_inversion}, we introduce an explicit elliptic Fourier integral operator $A$ in the complex domain, with the phase defined via an almost holomorphic extension of the weight $\Phi$, and obtain a $C^{\infty}$ symbol $a$, holomorphic to $\infty$--order along the anti-diagonal, as a solution of (\ref{eq1.2}). Let us observe that while the corresponding discussion in~\cite{DelHiSj} in the analytic case makes use of the existence of a microlocal inverse of the corresponding analytic Fourier integral operator, here the asymptotic inversion of $A$, with ${\cal O}(h^{\infty})$ errors, proceeds explicitly by considering the stationary phase expansion for $Aa$. Section \ref{sec:repr-property} is devoted to showing the approximate reproducing property for the operator $\widetilde{\Pi}_V$ in (\ref{eq1.3}), on the level of scalar products, $(\widetilde{\Pi}_Vu,v)_{L^2_{\Phi}(V)} = (u,v)_{H_{\Phi}(V)} +
{\cal O}(h^{\infty})$, for $u,v\in H_{\Phi}(V)$, with $v$ concentrated in a small neighborhood of $x_0$. Similar to~\cite{DelHiSj}, the proof depends on a resolution of the identity and a contour deformation argument, with some additional care required due to the lack of holomorphy in (\ref{eq1.3}). The proof of Theorem \ref{Theorem1} is then concluded in Section \ref{sec:end-proof}, making use of the $\overline{\partial}$ techniques. Finally, in Section \ref{secLG}, we recall, following~\cite{BBSj}, the link between the operator $\widetilde{\Pi}_V$ in Theorem \ref{Theorem1} and the orthogonal projection (\ref{eq1.1}), showing that the kernels of (\ref{eq1.1}) and (\ref{eq1.3}) are close pointwise, locally.

\medskip
\noindent
{\bf Acknowledgments}. We are very grateful to Alix Deleporte and Johannes Sj\"ostrand for very helpful discussions and for encouraging us to pursue the $C^{\infty}$ case.

\section{Asymptotic inversion of a Fourier integral operator}
\label{sec:asymp_inversion}
\setcounter{equation}{0}
The discussion in this section can be viewed as a natural analog in the $C^{\infty}$--setting of \cite[Section 3]{DelHiSj}, working systematically
with almost holomorphic extensions of the weights,~\cite{Ho69},~\cite{MelSj74}. Let $\Omega \subset \comp^n$ be open, and let $\Phi\in C^{\infty}(\Omega;\real)$ be strictly plurisubharmonic in $\Omega$,
\begeq
\label{eq2.1}
\sum_{j,k=1}^n \frac{\partial^2 \Phi}{\partial x_j \partial \overline{x}_k}(x) \xi_j \overline{\xi}_k \geq c(x) \abs{\xi}^2, \quad x\in \Omega, \quad \xi \in \comp^n,
\endeq
where $0 < c\in C(\Omega)$. Let $x_0 \in \Omega$. Identifying $\comp^n_x$ with the anti-diagonal $\{(x,y)\in \comp^{2n}; y = \overline{x}\}$, we see that exists $\Psi \in C^{\infty}({\rm neigh}((x_0,\overline{x_0}),\comp_{x,y}^{2n}))$ such that
\begeq
\label{eq2.2}
\Psi(x,\overline{x}) = \Phi(x), \quad x\in {\rm neigh}(x_0,\comp^n),
\endeq
and for every $N$,
\begeq
\label{eq2.3}
\left(\partial_{\overline{x}_j}\Psi\right)(x,y) = {\cal O}_N(\abs{y-\overline{x}}^N),\quad \left(\partial_{\overline{y}_j}\Psi\right)(x,y) = {\cal O}_N(\abs{y-\overline{x}}^N),\quad 1 \leq j \leq n,
\endeq
locally uniformly, see~\cite{Ho69},~\cite{Ma},~\cite{MelSj74},~\cite[Chapter 8]{DiSj}. Here we work with the usual operators,
\begeq
\label{eq2.3.1}
\partial_{\overline{x}_j} = \frac{1}{2}\left(\partial_{{\rm Re}\, x_j} + i \partial_{{\rm Im}\, x_j}\right), \quad \partial_{\overline{y}_j} = \frac{1}{2}\left(\partial_{{\rm Re}\, y_j} + i \partial_{{\rm Im}\, y_j}\right),\quad 1 \leq j \leq n.
\endeq
\begeq
\label{eq2.3.1.1}
\partial_{x_j} = \frac{1}{2}\left(\partial_{{\rm Re}\, x_j} - i \partial_{{\rm Im}\, x_j}\right), \quad \partial_{y_j} = \frac{1}{2}\left(\partial_{{\rm Re}\, y_j} - i \partial_{{\rm Im}\, y_j}\right),\quad 1 \leq j \leq n.
\endeq
We notice that (\ref{eq2.3}) and~\cite[Lemma X.2.2]{Treves} imply that for all $\alpha$, $\beta$, $\gamma$, $\delta\in \nat^n$ we have
\begeq
\label{eq2.3.2}
\left(D_x^{\alpha} D_{\overline{x}}^{\beta} D_y^{\gamma} D_{\overline{y}}^{\delta} \partial_{\overline{x}_j}\Psi\right)(x,y) = {\cal O}_{\alpha,\beta,\gamma, \delta, N}(\abs{y-\overline{x}}^N), \quad 1 \leq j \leq n,
\endeq
\begeq
\label{eq2.3.3}
\left(D_x^{\alpha} D_{\overline{x}}^{\beta} D_y^{\gamma} D_{\overline{y}}^{\delta} \partial_{\overline{y}_j}\Psi\right)(x,y) = {\cal O}_{\alpha,\beta,\gamma, \delta, N}(\abs{y-\overline{x}}^N),\quad 1\leq j \leq n,
\endeq
locally uniformly.

\medskip
\noindent
Our starting point is the following classical estimate, see for instance~\cite{BBSj}. Since related computations based on Taylor expansions will appear below, it will be natural to recall the proof.
\begin{prop}
\label{prop1}
We have
\begeq
\label{eq2.4}
\Phi(x) + \Phi(y) - 2{\rm Re}\, \Psi(x,\overline{y}) \asymp \abs{x-y}^2, \quad x,y\in {\rm neigh}(x_0,\comp^n).
\endeq
\end{prop}
\begin{proof}
Using (\ref{eq2.2}), (\ref{eq2.3.2}), (\ref{eq2.3.3}) we get
\begeq
\label{eq2.5}
(\partial _x\Psi )(x,\overline{x})=\partial _x\Phi (x),\ (\partial_y\Psi) (x,\overline{x})=\partial _{\overline{x}}\Phi (x),
\endeq
\begeq
\label{eq2.6}
\Psi''_{xx}(x,\overline{x}) = \Phi''_{xx}(x),\quad \Psi''_{xy}(x,\overline{x}) = \Phi''_{x\overline{x}}(x), \quad
\Psi''_{yy}(x,\overline{x}) = \Phi''_{\overline{x}\,\overline{x}}(x).
\endeq
By a Taylor expansion, we obtain then, using (\ref{eq2.2}), (\ref{eq2.3.2}), (\ref{eq2.3.3}), (\ref{eq2.5}), (\ref{eq2.6}),
\begin{multline*}
\Psi (x+z,\overline{x}+w)=\Phi (x)+\partial_x\Phi (x)\cdot z+\partial_{\overline{x}}\Phi (x)\cdot w \\
+ \frac{1}{2}\left(\Phi''_{xx}(x)z\cdot z + 2 \Phi''_{x\overline{x}}(x)w\cdot z + \Phi''_{\overline{x}\,\overline{x}}(x)w\cdot w\right) + {\cal O}(\abs{(z,w)}^3),
\end{multline*}
and therefore,
\begin{multline}
\label{eq2.7}
2\Re \Psi (x+z,\overline{x}+w)=2\Phi (x)+\partial _x\Phi (x)\cdot(z+\overline{w})+
\overline{\partial _x \Phi (x)}\cdot (\overline{z}+w) \\
+ \frac{1}{2} \left(2 {\rm Re}\, \left(\Phi''_{xx}(x)z\cdot z + \Phi''_{xx}(x)\overline{w}\cdot \overline{w}\right) + 2\Phi''_{x\overline{x}}(x)w\cdot z + 2\Phi''_{\overline{x}x}(x)\overline{w}\cdot\overline{z}\right) + {\cal O}(\abs{(z,w)}^3).
\end{multline}
Here we have used that $\partial_{\overline{x}}\Phi(x) = \overline{\partial_x \Phi(x)}$. Using also the Taylor expansions
\begin{multline}
\label{eq2.8}
\Phi (x+z)=\Phi (x)+\partial _x\Phi (x) \cdot z +\overline{\partial _x\Phi }\cdot \overline{z} \\
+ \frac{1}{2} \left(\Phi''_{xx}(x)z\cdot z + 2 \Phi''_{x\overline{x}}(x)\overline{z}\cdot z + \Phi''_{\overline{x}\,\overline{x}}(x) \overline{z}\cdot \overline{z}\right) + {\cal O}(\abs{z}^3),
\end{multline}
\begin{multline}
\label{eq2.9}
\Phi (x+\overline{w})=\Phi (x)+\partial _x\Phi (x) \cdot \overline{w}+\overline{\partial_x\Phi }(x)\cdot w \\
+ \frac{1}{2} \left(\Phi''_{xx}(x)\overline{w}\cdot \overline{w} + 2 \Phi''_{x\overline{x}}(x)w\cdot \overline{w} + \Phi''_{\overline{x}\,\overline{x}}(x)w\cdot w\right) + {\cal O}(\abs{w}^3),
\end{multline}
we get from (\ref{eq2.7}), (\ref{eq2.8}), and (\ref{eq2.9}),
\begeq
\label{eq2.10}
\Phi (x+z)+\Phi (x+\overline{w})-2\Re \Psi (x+z,\overline{x}+w) = \Phi''_{x\overline{x}}(x)(w-\overline{z})\cdot (\overline{w}-z) + {\cal O}(\abs{(z,w)}^3).
\endeq
Hence we get for $x,y\in {\rm neigh}(x_0,\comp^n)$,
\begin{equation}
\label{eq2.11}
\Phi (x)+\Phi (y) -2\Re \Psi (x,\overline{y}) = \Phi''_{x\overline{x}}(x_0)(\overline{y} - \overline{x})\cdot (y-x) +
{\cal O}(\abs{y-x}^3),
\end{equation}
and using the strict plurisubharmonicity of $\Phi$ given in (\ref{eq2.1}), we infer (\ref{eq2.4}).
\end{proof}

\bigskip
\noindent
Let us set, following~\cite[Section 3]{DelHiSj},
\begeq
\label{eq2.12}
\varphi(y,\widetilde{x};x,\widetilde{y}) = \Psi(x,\widetilde{y}) - \Psi(x,\widetilde{x}) - \Psi(y,\widetilde{y}) + \Psi(y,\widetilde{x}).
\endeq
We have $\varphi \in C^{\infty}({\rm neigh}((x_0,\overline{x_0};x_0,\overline{x_0}),\comp^{4n}))$. Of particular interest for us here are critical points of $(x,\widetilde{y}) \mapsto \varphi(y,\widetilde{x};x,\widetilde{y})$.

\begin{prop}
\label{prop2}
For each $(y,\widetilde{x}) \in {\rm neigh}((x_0,\overline{x_0}),\comp^{2n})$, the complex valued $C^{\infty}$ function
$$
{\rm neigh}((x_0,\overline{x_0}),\comp^{2n})\ni (x,\widetilde{y}) \mapsto \varphi(y,\widetilde{x};x,\widetilde{y})
$$
has a unique critical point given by $(x,\widetilde{y}) = (y,\widetilde{x})$. The corresponding critical value is equal to $0$, and when $\widetilde{x} = \overline{y}$, the quadratic part of the Taylor expansion of the $C^{\infty}$ function $\comp^{2n} \ni (z,w)\mapsto \varphi(y,\widetilde{x};y+z,\widetilde{x}+w)$ at $(z,w) = (0,0)$ is the non-degenerate holomorphic quadratic form on $\comp^{2n}$ given by
$(z,w)\mapsto \Phi''_{x\overline{x}}(y)w\cdot z$.
\end{prop}
\begin{proof}
We shall consider the Taylor expansion of the $C^{\infty}$ function $(z,w)\mapsto \varphi(y,\widetilde{x};y+z,\widetilde{x}+w)$ at $(z,w) = (0,0)\in \comp^{2n}$. Here $(y,\widetilde{x}) \in {\rm neigh}((x_0,\overline{x_0}),\comp^{2n})$. Let us write, using Taylor's formula and the almost holomorphy of $\Psi$ along the anti-diagonal, as in (\ref{eq2.3}), (\ref{eq2.3.2}), (\ref{eq2.3.3}),
\begin{multline}
\label{eq2.17}
\Psi(y+z,\widetilde{x} + w) = \Psi(y,\widetilde{x}) + \Psi'_x(y,\widetilde{x})\cdot z + \Psi'_{\overline{x}}(y,\widetilde{x})\cdot \overline{z}
+ \Psi'_y(y,\widetilde{x})\cdot w + \Psi'_{\overline{y}}(y,\widetilde{x})\cdot \overline{w} \\
+ \frac{1}{2}\left(\Psi''_{xx}(y,\widetilde{x})z\cdot z + 2 \Psi''_{xy}(y,\widetilde{x})w\cdot z + \Psi''_{yy}(y,\widetilde{x})w\cdot w\right) \\
+ {\cal O}_N(\abs{\overline{y}-\widetilde{x}}^N)(\abs{(z,w)}^2) + {\cal O}(\abs{(z,w)}^3),
\end{multline}
\begin{multline}
\label{eq2.18}
\Psi(y,\widetilde{x} + w) = \Psi(y,\widetilde{x}) + \Psi'_y(y,\widetilde{x})\cdot w + \Psi'_{\overline{y}}(y,\widetilde{x})\cdot \overline{w} \\
+ \frac{1}{2} \Psi''_{yy}(y,\widetilde{x})w\cdot w + {\cal O}_N(\abs{\overline{y}-\widetilde{x}}^N)\abs{w}^2 + {\cal O}(\abs{w}^3).
\end{multline}
\begin{multline}
\label{eq2.19}
\Psi(y+z,\widetilde{x}) = \Psi(y,\widetilde{x}) + \Psi'_x(y,\widetilde{x})\cdot z + \Psi'_{\overline{x}}(y,\widetilde{x})\cdot \overline{z} \\
+ \frac{1}{2} \Psi''_{xx}(y,\widetilde{x})z\cdot z + {\cal O}_N(\abs{\overline{y}-\widetilde{x}}^N)\abs{z}^2 + {\cal O}(\abs{z}^3).
\end{multline}
Here $N\in \nat$ is arbitrary. We get, using (\ref{eq2.12}), (\ref{eq2.17}), (\ref{eq2.18}), (\ref{eq2.19}),
\begeq
\label{eq2.20}
\varphi(y,\widetilde{x};y+z,\widetilde{x}+w) = \Psi''_{xy}(y,\widetilde{x})w\cdot z + {\cal O}_N(\abs{\overline{y}-\widetilde{x}}^N)(\abs{(z,w)}^2) + {\cal O}(\abs{(z,w)}^3).
\endeq
This shows in particular that $(z,w) = (0,0)$ is a critical point of $(z,w)\mapsto \varphi(y,\widetilde{x};y+z,\widetilde{x}+w)$, with the corresponding critical value being equal to $0$. We notice also that the matrix $\Psi''_{xy}(y,\widetilde{x})$ is invertible for $(y,\widetilde{x})\in {\rm neigh}((x_0,\overline{x_0}),\comp^{2n})$. Restricting $(y,\widetilde{x})$ in (\ref{eq2.20}) to the anti-diagonal in $\comp^{2n}$, i.e. letting $\widetilde{x} = \overline{y}$, and using (\ref{eq2.6}), we get,
\begeq
\label{eq2.21}
\varphi(y,\overline{y};y+z,\overline{y}+w) = \Phi''_{x\overline{x}}(y) w\cdot z + {\cal O}(\abs{(z,w)}^3).
\endeq
Here, in view of (\ref{eq2.1}), the holomorphic quadratic form $(z,w)\mapsto \Phi''_{x\overline{x}}(y) w\cdot z$ is non-degenerate on $\comp^{2n}_{z,w}$, for $y\in {\rm neigh}(x_0,\comp^n)$, and this completes the proof.
\end{proof}

\bigskip
\noindent
It follows from the last observation in the proof of Proposition \ref{prop2} that the pluriharmonic quadratic form $q(z,w) := {\rm Re}\, \left(\Phi''_{x\overline{x}}(x_0) w\cdot z\right)$ is non-degenerate on $\comp^{2n}_{z,w}$, and hence necessarily of signature $(2n,2n)$. Explicitly, we have
\begeq
\label{eq2.22}
q(z,\overline{z}) = {\rm Re}\, \left(\Phi''_{x\overline{x}}(x_0) \overline{z} \cdot z\right) \asymp \abs{z}^2, \quad z\in \comp^n,
\endeq
and therefore
\begeq
\label{eq2.23}
q(iz,i\overline{z}) = - q(z,\overline{z}) \asymp -\abs{z}^2,\quad z\in \comp^n.
\endeq
Using the continuity of $\Psi''_{xy}(y,\widetilde{x})$ we conclude that for all $(y,\widetilde{x})\in \comp^{2n}$ sufficiently close to $(x_0,\overline{x_0})$, we have
\begeq
\label{eq2.24}
{\rm Re}\, \left(\Psi''_{xy}(y,\widetilde{x})\overline{z} \cdot z\right) \asymp \abs{z}^2, \quad z\in \comp^n,
\endeq
\begeq
\label{eq2.25}
{\rm Re}\, \left(\Psi''_{xy}(y,\widetilde{x})i\overline{z} \cdot iz\right) \asymp -\abs{z}^2, \quad z\in \comp^n.
\endeq
Combining (\ref{eq2.20}) with (\ref{eq2.24}), (\ref{eq2.25}), we get for all $(y,\widetilde{x}) \in {\rm neigh}((x_0,\overline{x_0}),\comp^{2n})$,
\begin{multline}
\label{eq2.26}
{\rm Re}\, \varphi(y,\widetilde{x};y+z,\widetilde{x}+\overline{z}) = {\rm Re}\, \left(\Psi''_{xy}(y,\widetilde{x})\overline{z}\cdot z\right) + {\cal O}_N(\abs{\overline{y}-\widetilde{x}}^N)\abs{z}^2 + {\cal O}(\abs{z}^3) \\
\geq \frac{1}{C}\abs{z}^2,\quad z\in {\rm neigh}(0,\comp^n),
\end{multline}
\begin{multline}
\label{eq2.27}
{\rm Re}\, \varphi(y,\widetilde{x};y+iz,\widetilde{x}+i\overline{z}) = {\rm Re}\, \left(\Psi''_{xy}(y,\widetilde{x})i\overline{z}\cdot iz\right) + {\cal O}_N(\abs{\overline{y}-\widetilde{x}}^N)\abs{z}^2 + {\cal O}(\abs{z}^3) \\
\leq -\frac{1}{C}\abs{z}^2,\quad z\in {\rm neigh}(0,\comp^n),
\end{multline}
It follows from (\ref{eq2.20}), (\ref{eq2.26}), (\ref{eq2.27}) that for each $(y,\widetilde{x}) \in {\rm neigh}((x_0,\overline{x_0}),\comp^{2n})$, the critical point $(x,\widetilde{y}) = (y,\widetilde{x})$ of the real-valued $C^{\infty}$ function ${\rm neigh}((x_0,\overline{x_0}),\comp^{2n})\ni (x,\widetilde{y}) \mapsto {\rm Re}\, \varphi(y,\widetilde{x};x,\widetilde{y})$ is non-degenerate of signature $(2n,2n)$.

\bigskip
\noindent
Let $\Gamma(y,\widetilde{x}) \subset \comp^{2n}_{x,\widetilde{y}}$ be a smooth $2n$-dimensional contour of integration passing through the critical point $(x,\widetilde{y}) = (y,\widetilde{x})$ and depending smoothly on
$(y,\widetilde{x})\in {\rm neigh}((x_0,\overline{x_0}),\comp^{2n})$, such that along $\Gamma(y,\widetilde{x})$, we have
\begeq
\label{eq2.28}
{\rm Re}\, \varphi(y,\widetilde{x};x,\widetilde{y}) \leq -\frac{1}{C} {\rm dist}\left((x,\widetilde{y}),(y,\widetilde{x})\right)^2.
\endeq
Following~\cite[Chapter 3]{Sj82}, we shall say that $\Gamma(y,\widetilde{x})$ is a good contour for the $C^{\infty}$ real-valued function $(x,\widetilde{y}) \mapsto {\rm Re}\, \varphi(y,\widetilde{x};x,\widetilde{y})$. In particular, it follows from (\ref{eq2.27}) that the $2n$-dimensional affine contour
\begeq
\label{eq2.29}
\Gamma(y,\widetilde{x}): {\rm neigh}(0,\comp^n) \ni z \mapsto (y+z,\widetilde{x}-\overline{z}) \in \comp^{2n}_{x,\widetilde{y}}
\endeq
is good.

\bigskip
\noindent
Proceeding similarly to~\cite[Section 3]{DelHiSj}, we shall now introduce a suitable Fourier integral operator in the complex domain, with the function $\varphi$ in (\ref{eq2.12}) playing the role of the phase function, with no fiber variables present. To this end, let us first specify a suitable class of amplitudes. Let $a\in S^0_{{\rm cl}}({\rm neigh}((x_0,\overline{x_0}),\comp^{2n}))$,
\begeq
\label{eq2.29.1}
a(x,\widetilde{y};h) \sim \sum_{j=0}^{\infty} a_j(x,\widetilde{y}) h^j, \quad h\rightarrow 0^+,
\endeq
be a classical $C^{\infty}$ symbol, with the asymptotic expansion (\ref{eq2.29.1}) in $C^{\infty}({\rm neigh}((x_0,\overline{x_0}),\comp^{2n}))$, such that $a_j \in C^{\infty}({\rm neigh}((x_0,\overline{x_0}),\comp^{2n}))$ satisfy
\begeq
\label{eq2.30}
\left(\partial_{\overline{x}}a_j\right)(x,\widetilde{y}) = {\cal O}(\abs{\widetilde{y}-\overline{x}}^{\infty}),\quad \left(\partial_{\overline{\widetilde{y}}}a_j\right)(x,\widetilde{y}) =
{\cal O}(\abs{\widetilde{y}-\overline{x}}^{\infty}),\quad j=0,1,2,\ldots.
\endeq

\medskip
\noindent
Given a good contour $\Gamma(y,\widetilde{x}) \subset \comp^{2n}_{x,\widetilde{y}}$ for the $C^{\infty}$ function $(x,\widetilde{y}) \mapsto {\rm Re}\, \varphi(y,\widetilde{x};x,\widetilde{y})$ and an amplitude $a(x,\widetilde{y};h)$ satisfying (\ref{eq2.29.1}), (\ref{eq2.30}), we set
\begeq
\label{eq2.31}
(A_{\Gamma} a)(y,\widetilde{x};h) = \frac{1}{h^n}\int\!\!\!\int_{\Gamma(y,\widetilde{x})} e^{\frac{2}{h} \varphi(y,\widetilde{x};x,\widetilde{y})} a(x,\widetilde{y};h)\, dx\, d\widetilde{y}.
\endeq
We have $A_{\Gamma}a \in C^{\infty}({\rm neigh}((x_0,\overline{x_0}),\comp^{2n})$, and let us first check that the definition of $A_{\Gamma}a$ is essentially independent of the choice of a good contour, up to a rapidly vanishing error as $h\rightarrow 0^+$, provided that $(y,\widetilde{x})$ is confined to the anti-diagonal.

\begin{prop}
\label{prop_contour}
There exists an open neighborhood $V_0 \Subset \Omega \subset \comp^n$ of $x_0$ such that for any two good contours $\Gamma(y,\overline{y})$, $\Gamma_0(y,\overline{y})$ for the function $(x,\widetilde{y}) \mapsto {\rm Re}\, \varphi(y,\overline{y};x,\widetilde{y})$, for $y\in V_0$, and any amplitude $a$ satisfying {\rm (\ref{eq2.29.1})}, {\rm (\ref{eq2.30})}, we have for $y\in V_0$,
\begeq
\label{eq2.32}
(A_{\Gamma}a)(y,\overline{y};h) - (A_{\Gamma_0}a)(y,\overline{y};h) = {\cal O}(h^{\infty}),
\endeq
in the $C^{\infty}(V_0)$ sense.
\end{prop}
\begin{proof}
Let us start by making some general remarks concerning good contours when parameters are present, closely related to the discussion in~\cite[Chapter 3]{Sj82},~\cite[Chapter 1]{Delort}. Let $f(x,y)$, $x\in \real^n$, $y\in \real^{2N}$, be a real-valued $C^{\infty}$ function in a neighborhood of $(0,0)\in \real^{n}_x \times \real^{2N}_y$. Assume that $f'_y(0,0) = 0$ and that $f''_{yy}(0,0)$ is non-degenerate of signature $(N,N)$. The implicit function theorem gives that the equation $f'_y(x,y) = 0$ uniquely defines a $C^{\infty}$ function $y = y(x)$ in a neighborhood of $0$, with $y(0) = 0$, and applying the Morse lemma with parameters~\cite[Appendix C]{H_book}, we obtain that there exist new $C^{\infty}$ coordinates $z = z(y) = z(x,y)$ for $\real^{2N}_y$ near $y=0$ when $x\in \real^n$ is small, such that writing $z = (t,s)$ with $t,s\in \real^{N}$, we have
\begeq
\label{eq2.33}
f(x,y) = f(x,y(x)) + \frac{1}{2}(t^2 - s^2).
\endeq
We may also recall from~\cite[Appendix C]{H_book} that the new coordinates, which also depend on $x$, are centered at the critical point $y(x)$ and are of the form
\begeq
\label{eq2.34}
z = Q(x,y)(y-y(x)),
\endeq
where the matrix $Q(0,0)$ is invertible. Assume next that $\Gamma(x)\subset \real^{2N}_y$ is a good contour for $y\mapsto f(x,y)$, depending smoothly on $x$, so that $\Gamma(x)$ passes through $y(x)$ and along $\Gamma(x)$, we have for all $x\in \real^{n}$ small enough,
\begeq
\label{eq2.35}
f(x,y)\leq f(x,y(x)) - \frac{1}{C}\abs{y - y(x)}^2,\quad y\in \Gamma(x),
\endeq
for some $C>0$. It follows from (\ref{eq2.33}), (\ref{eq2.34}), and (\ref{eq2.35}) that along $\Gamma(x)$, we have
\begeq
\label{eq2.36}
\frac{1}{2}(t^2 - s^2) \leq - \frac{1}{{\cal O}(1)}\left(t^2 + s^2\right),
\endeq
and by the implicit function theorem, we obtain therefore that in the Morse coordinates $z = (t,s)$, the contour $\Gamma(x)$ takes the form
\begeq
\label{eq2.37}
t = g(s,x),\quad s\in {\rm neigh}(0,\real^N),
\endeq
where
\begeq
\label{eq2.37.1}
\abs{g(s,x)} \leq \alpha \abs{s}, \quad \alpha < 1.
\endeq
See also~\cite[Chapter 3]{Sj82}. Throughout the discussion above, $x\in \real^n$ varies in a sufficiently small neighborhood of the origin. Letting $z \mapsto y(z)$ be the inverse of the map $y \mapsto z(y)$, well defined for $x$ small, we obtain the following parametrization of the good contour $\Gamma(x)$,
\begeq
\label{eq2.38}
{\rm neigh}(0,\real^N) \ni s \mapsto y(g(s,x),s) \in \real^{2N}_y.
\endeq
A consequence of this discussion is that if $\Gamma_0(x) \subset \real^{2N}_y$ is another good contour for $y\mapsto f(x,y)$, depending smoothly on $x$, then its representation in the Morse coordinates $z = (t,s)$ takes the form
\begeq
\label{eq2.39}
t = g_0(s,x), \quad s\in {\rm neigh}(0,\real^N),
\endeq
\begeq
\label{eq2.40}
\abs{g_0(s,x)} \leq \alpha \abs{s}, \quad \alpha < 1,
\endeq
and the contours $\Gamma(x)$, $\Gamma_0(x)$ are therefore homotopic via the deformation
\begeq
\label{eq2.41}
{\rm neigh}(0,\real^N)\times [0,1] \ni (s,\theta) \mapsto y(\theta g(s,x) + (1-\theta)g_0(s,x),s),
\endeq
well defined for all $x$ small enough, with the $C^{\infty}$ dependence on $x$. It follows also from (\ref{eq2.33}), (\ref{eq2.37}), (\ref{eq2.37.1}), (\ref{eq2.39}), (\ref{eq2.40}) that when $\theta\in [0,1]$, the contour
\begeq
\label{eq2.42}
\Gamma_{\theta}(x): {\rm neigh}(0,\real^N) \ni s \mapsto y(\theta g(s,x) + (1-\theta)g_0(s,x),s),
\endeq
is good for $y\mapsto f(x,y)$, uniformly in $\theta\in [0,1]$ (and $x\in \real^n$ small enough).

\medskip
\noindent
We shall now apply the discussion above to the $C^{\infty}$ function
\begeq
\label{eq2.43}
{\rm neigh}((x_0,\overline{x_0}),\comp^{2n}) \ni (x,\widetilde{y}) \mapsto {\rm Re}\, \varphi(y,\overline{y};x,\widetilde{y}),
\endeq
with $y\in {\rm neigh}(x_0,\comp^n)$ playing the role of the parameters. Let $\Gamma = \Gamma(y)$, $\Gamma_0 = \Gamma_0(y)$ be two good contours for the function in (\ref{eq2.43}), and let $\Gamma_{\theta}(y) \subset \comp^{2n}_{x,\widetilde{y}}$ be the "intermediate" contour defined as in (\ref{eq2.42}), for $\theta\in [0,1]$, where $N = 2n$. Letting $G_{[0,1]}(y)\subset \comp^{2n}_{x,\widetilde{y}}$ be the $(2n+1)$--dimensional contour formed by the union of the contours $\Gamma_{\theta}(y)$, for $\theta \in [0,1]$, parametrized as in (\ref{eq2.41}), we may write by an application of Stokes' formula to the $(2n,0)$--differential form $\omega = \displaystyle e^{\frac{2}{h} \varphi(y,\overline{y};x,\widetilde{y})} a(x,\widetilde{y};h)\, dx\wedge d\widetilde{y}$,
\begin{multline}
\label{eq2.44}
\int\!\!\!\int_{\partial G_{[0,1]}(y)} \omega = \int\!\!\!\int\!\!\!\int_{G_{[0,1]}(y)} d\omega =
\int\!\!\!\int\!\!\!\int_{G_{[0,1]}(y)} d\left(e^{\frac{2}{h} \varphi(y,\overline{y};x,\widetilde{y})} a(x,\widetilde{y};h)\right)\wedge dx\wedge d\widetilde{y} \\
= \int\!\!\!\int\!\!\!\int_{G_{[0,1]}(y)} \partial_{\overline{x}} \left(e^{\frac{2}{h} \varphi(y,\overline{y};x,\widetilde{y})} a(x,\widetilde{y};h)\right)\wedge dx\wedge d\widetilde{y} \\
+ \int\!\!\!\int\!\!\!\int_{G_{[0,1]}(y)} \partial_{\overline{\widetilde{y}}} \left(e^{\frac{2}{h} \varphi(y,\overline{y};x,\widetilde{y})} a(x,\widetilde{y};h)\right)\wedge dx\wedge d\widetilde{y}.
\end{multline}
Using (\ref{eq2.12}) we compute, for $1\leq j \leq n$,
\begin{multline}
\label{eq2.45}
\partial_{\overline{x}_j} \left(e^{\frac{2}{h} \varphi(y,\overline{y};x,\widetilde{y})} a(x,\widetilde{y};h)\right) =
e^{\frac{2}{h} \varphi(y,\overline{y};x,\widetilde{y})}\left(\frac{2}{h} \partial_{\overline{x}_j} \varphi(y,\overline{y};x,\widetilde{y}) a(x,\widetilde{y};h) + \partial_{\overline{x}_j} a(x,\widetilde{y};h)\right)\\
= e^{\frac{2}{h} \varphi(y,\overline{y};x,\widetilde{y})} \left(\frac{2}{h}\partial_{\overline{x}_j}\left(\Psi(x,\widetilde{y}) - \Psi(x,\overline{y})\right)a(x,\widetilde{y};h) + \partial_{\overline{x}_j} a(x,\widetilde{y};h)\right),
\end{multline}
and here we have in view of (\ref{eq2.3}), for all $N$,
\begin{multline}
\label{eq2.46}
\abs{\partial_{\overline{x}_j}\left(\Psi(x,\widetilde{y}) - \Psi(x,\overline{y})\right)} \leq {\cal O}_N(1) \left(\abs{\overline{x} - \widetilde{y}}^N + \abs{x-y}^N\right)\\
\leq {\cal O}_N(1) \left(\abs{\overline{x} - \overline{y}}^N + \abs{\overline{y} - \widetilde{y}}^N + \abs{x-y}^N\right) \leq
{\cal O}_N(1) \left(\abs{\overline{y} - \widetilde{y}}^N + \abs{x-y}^N\right) \\
\leq {\cal O}_N(1) {\rm dist}\, ((x,\widetilde{y}),(y,\overline{y}))^N.
\end{multline}
Furthermore, using (\ref{eq2.29.1}) and (\ref{eq2.30}) we get
\begeq
\label{eq2.47}
\abs{\partial_{\overline{x}_j} a(x,\widetilde{y};h)} \leq {\cal O}_N(1)\left(\abs{\overline{x} - \widetilde{y}}^N + h^N\right) \leq
{\cal O}_N(1) \left({\rm dist}\, ((x,\widetilde{y}),(y,\overline{y}))^N + h^N\right).
\endeq
We get therefore using (\ref{eq2.28}), (\ref{eq2.45}), (\ref{eq2.46}), (\ref{eq2.47}), uniformly along $G_{[0,1]}(y)$,
\begin{multline}
\label{eq2.48}
\abs{\partial_{\overline{x}_j} \left(e^{\frac{2}{h} \varphi(y,\overline{y};x,\widetilde{y})} a(x,\widetilde{y};h)\right)} \leq
{\cal O}_N(1) e^{\frac{2}{h} {\rm Re}\, \varphi(y,\overline{y};x,\widetilde{y})}\left(\frac{1}{h}{\rm dist}\, ((x,\widetilde{y}),(y,\overline{y}))^N + h^N\right)\\
{\cal O}_N(1) e^{-\frac{1}{Ch}{\rm dist}\, ((x,\widetilde{y}),(y,\overline{y}))^2}\left(\frac{1}{h}{\rm dist}\, ((x,\widetilde{y}),(y,\overline{y}))^N + h^N\right).
\end{multline}
Using that
$$
e^{-t^2/Ch} t^N \leq {\cal O}_N(1) h^{N/2},\quad t\geq 0,\,\,N =1,2,\ldots,
$$
we conclude that we have uniformly along $G_{[0,1]}(y)$, for all $N$,
\begeq
\label{eq2.49}
\abs{\partial_{\overline{x}} \left(e^{\frac{2}{h} \varphi(y,\overline{y};x,\widetilde{y})} a(x,\widetilde{y};h)\right)} \leq{\cal O}_N(1) h^N.
\endeq
This bound is uniform in $y\in {\rm neigh}(x_0,\comp^n)$. Next, when considering
\begin{multline}
\label{eq2.50}
\partial_{\overline{\widetilde{y}}} \left(e^{\frac{2}{h} \varphi(y,\overline{y};x,\widetilde{y})} a(x,\widetilde{y};h)\right) =
e^{\frac{2}{h} \varphi(y,\overline{y};x,\widetilde{y})}\left(\frac{2}{h} \partial_{\overline{\widetilde{y}}} \varphi(y,\overline{y};x,\widetilde{y}) a(x,\widetilde{y};h) + \partial_{\overline{\widetilde{y}}} a(x,\widetilde{y};h)\right)\\
= e^{\frac{2}{h} \varphi(y,\overline{y};x,\widetilde{y})} \left(\frac{2}{h}\partial_{\overline{\widetilde{y}}}\left(\Psi(x,\widetilde{y}) - \Psi(y,\widetilde{y})\right)a(x,\widetilde{y};h) + \partial_{\overline{\widetilde{y}}} a(x,\widetilde{y};h)\right),
\end{multline}
we write similarly to (\ref{eq2.46}), for $N=1,2,\ldots$,
\begin{multline}
\label{eq2.51}
\abs{\partial_{\overline{\widetilde{y}}}\left(\Psi(x,\widetilde{y}) - \Psi(y,\widetilde{y})\right)} \leq {\cal O}_N(1) \left(\abs{\overline{x} - \widetilde{y}}^N + \abs{\overline{y} - \widetilde{y}}^N\right)\\
\leq {\cal O}_N(1) \left(\abs{\overline{x} - \overline{y}}^N + \abs{\overline{y} - \widetilde{y}}^N +  \abs{\overline{y} - \widetilde{y}}^N \right) \leq
{\cal O}_N(1) \left(\abs{x-y}^N + \abs{\overline{y} - \widetilde{y}}^N\right) \\
\leq {\cal O}_N(1) {\rm dist}\, ((x,\widetilde{y}),(y,\overline{y}))^N.
\end{multline}
Using also (\ref{eq2.29.1}), (\ref{eq2.30}), we conclude that
\begin{multline}
\label{eq2.52}
\abs{\partial_{\overline{\widetilde{y}}} \left(e^{\frac{2}{h} \varphi(y,\overline{y};x,\widetilde{y})} a(x,\widetilde{y};h)\right)} \leq
\\
{\cal O}_N(1) e^{-\frac{1}{Ch}{\rm dist}\, ((x,\widetilde{y}),(y,\overline{y}))^2}\left(\frac{1}{h}{\rm dist}\, ((x,\widetilde{y}),(y,\overline{y}))^N + h^N\right),
\end{multline}
and therefore, similar to (\ref{eq2.49}), we get uniformly along $G_{[0,1]}(y)$, for all $N$,
\begeq
\label{eq2.53}
\abs{\partial_{\overline{\widetilde{y}}} \left(e^{\frac{2}{h} \varphi(y,\overline{y};x,\widetilde{y})} a(x,\widetilde{y};h)\right)} \leq{\cal O}_N(1) h^N.
\endeq
We get, combining (\ref{eq2.44}), (\ref{eq2.49}), and (\ref{eq2.53}),
\begeq
\label{eq2.54}
\int\!\!\!\int_{\partial G_{[0,1]}(y)} e^{\frac{2}{h} \varphi(y,\overline{y};x,\widetilde{y})} a(x,\widetilde{y};h)\, dx\wedge d\widetilde{y} = {\cal O}(h^{\infty}),
\endeq
uniformly for $y\in {\rm neigh}(x_0,\comp^n)$. Here we may write, with a suitable orientation,
$$
\partial G_{[0,1]}(y) = \Gamma(y) - \Gamma_0(y) + \Gamma_1(y),
$$
where
\begeq
\label{eq2.55}
{\rm Re}\, \varphi(y,\overline{y};x,\widetilde{y}) \leq -\frac{1}{C},\quad (x,\widetilde{y})\in \Gamma_1(y),
\endeq
for some $C>0$, when $y\in {\rm neigh}(x_0,\comp^n)$. It follows that
\begeq
\label{eq2.56}
(A_{\Gamma}a)(y,\overline{y};h) - (A_{\Gamma_0}a)(y,\overline{y};h) = {\cal O}(h^{\infty}),
\endeq
uniformly for $y\in {\rm neigh}(x_0,\comp^n)$. Let us see, finally, that the relation (\ref{eq2.56}) holds in the $C^{\infty}$ sense, i.e. also for the derivatives of $A_{\Gamma}a - A_{\Gamma_0}a$. To this end, we observe first that for all $\alpha,\beta \in \nat^n$ there exists $M_{\alpha\beta}\geq 0$ such that
\begeq
\label{eq2.57}
\partial^{\alpha}_y \partial^{\beta}_{\overline{y}}\left(A_{\Gamma}a - (A_{\Gamma_0}a)\right)(y,\overline{y};h) = {\cal O}_{\alpha}(1) h^{-M_{\alpha\beta}}.
\endeq
Combining (\ref{eq2.56}), (\ref{eq2.57}) with the convexity estimates for the derivatives of a smooth function~\cite[Chapter 1]{GrSj}, we conclude that
\begeq
\label{eq2.58}
\partial^{\alpha}_y\partial^{\beta}_{\overline{y}}\left(A_{\Gamma}a - (A_{\Gamma_0}a)\right)(y,\overline{y};h) = {\cal O}(h^{\infty}),
\endeq
uniformly, after an arbitrarily small decrease of the neighborhood of $x_0\in \comp^n$ where (\ref{eq2.56}) holds. The proof is complete.
\end{proof}

\bigskip
\noindent
We shall next proceed to establish the existence of a complete asymptotic expansion for $(A_{\Gamma}a)(y,\overline{y};h)$, as $h\rightarrow 0^+$. When doing so, thanks to Proposition \ref{prop_contour}, it will be convenient to work with the particular choice of the good contour $\Gamma(y,\overline{y})$ given in (\ref{eq2.29}). Using the parametrization of $\Gamma(y,\overline{y})$ given in (\ref{eq2.29}) and (\ref{eq2.31}), we get
\begeq
\label{eq2.59}
A_{\Gamma}a(y,\overline{y};h) = \frac{C_n}{h^n} \int_U e^{\frac{i}{h} f(y,z)} b(y,z;h)\, L(dz)
\endeq
Here $f(y,z) = -2i \varphi(y,\overline{y}; y+z,\overline{y}-\overline{z})$, $b(y,z;h) = a(y+z,\overline{y} - \overline{z};h)$, and $L(dz)$ is the Lebesgue measure on $\comp^n$. Furthermore, the constant $C_n\neq 0$ in (\ref{eq2.59}) depends on the dimension $n$ only and the region of integration $U\subset \comp^n$ is a small neighborhood of the origin. Using (\ref{eq2.21}) we see that
\begeq
\label{eq2.60}
f(y,z) = 2i \Phi''_{x\overline{x}}(y)\overline{z}\cdot z + {\cal O}(\abs{z}^3),
\endeq
so that in particular
\begeq
\label{eq2.61}
{\rm Im}\, f(y,z)\geq \frac{1}{C}\abs{z}^2,\quad z\in U,
\endeq
for some $C>0$ and all $y\in \comp^n$ close enough to $x_0$. For future reference, we shall now proceed to compute ${\rm det}\, \left(\nabla^2_z f(y,0)/i\right)$, where the Hessian $\nabla^2_z$ is taken in the real sense of $\real^{2n}\simeq \comp^n$, so that $\nabla^2_z f(y,0)/i$ is a real symmetric $2n\times 2n$ matrix. Writing $\comp^n \ni z = t + is$, $t,s\in \real^n$, we see that the quadratic part of the Taylor expansion of $z\mapsto f(y,z)/i$ at the origin, given in (\ref{eq2.60}), is of the form
\begin{multline}
\label{eq2.62}
2\Phi''_{x\overline{x}}(y)\overline{z}\cdot z = 2\left(\Phi''_{x\overline{x}}(y)t\cdot t + \Phi''_{x\overline{x}}(y)s\cdot s + i\Phi''_{x\overline{x}}(y)t\cdot s - i\Phi''_{x\overline{x}}(y)s\cdot t\right)\\
= 2\left(A_1 t\cdot t + A_1s\cdot s - 2A_2 t\cdot s\right).
\end{multline}
Here we have written $\Phi''_{x\overline{x}}(y) = A_1 + iA_2$, with $A_1$, $A_2$ being $n\times n$ real matrices and observed that since $\Phi''_{x\overline{x}}(y)$ is Hermitian, we have $A_1^t = A_1$, $A_2^t = -A_2$. We get
\begeq
\label{eq2.63}
2\Phi''_{x\overline{x}}(y)\overline{z}\cdot z = 2 \begin{pmatrix}
A_1 & A_2 \\\
-A_2 & A_1
\end{pmatrix} \begin{pmatrix}
t \\\
s
\end{pmatrix}\cdot \begin{pmatrix}
t \\\
s
\end{pmatrix},
\endeq
and writing
\begeq
\label{eq2.63.1}
\Phi''_{x\overline{x}}(y)\overline{z}\cdot z = \frac{1}{2} \begin{pmatrix}
\Phi''_{x\overline{x}}(y) & 0 \\\
0 & \Phi''_{\overline{x}x}(y)
\end{pmatrix} \begin{pmatrix}
\overline{z} \\\
z
\end{pmatrix}\cdot \begin{pmatrix}
z \\\
\overline{z}
\end{pmatrix},
\endeq
\begeq
\label{eq2.63.2}
\begin{pmatrix}
z \\\
\overline{z}
\end{pmatrix} = \begin{pmatrix}
1 & i \\\
1 & -i
\end{pmatrix} \begin{pmatrix}
t \\\
s
\end{pmatrix}, \quad \begin{pmatrix}
\overline{z} \\\
z
\end{pmatrix} = \begin{pmatrix}
1 & -i \\\
1 & i
\end{pmatrix} \begin{pmatrix}
t \\\
s
\end{pmatrix},
\endeq
we obtain the factorization
\begeq
\label{eq2.63.3}
\frac{1}{2} \begin{pmatrix}
1 & 1 \\\
i & -i
\end{pmatrix} \begin{pmatrix}
\Phi''_{x\overline{x}}(y) & 0 \\\
0 & \Phi''_{\overline{x}x}(y)
\end{pmatrix} \begin{pmatrix}
1 & -i \\\
1 & i
\end{pmatrix} = \begin{pmatrix}
A_1 & A_2 \\\
-A_2 & A_1
\end{pmatrix}.
\endeq
It follows from (\ref{eq2.60}), (\ref{eq2.63}), and (\ref{eq2.63.3}) that
\begeq
\label{eq2.64}
{\rm det}\, \left(\frac{\nabla^2_z f(y,0)}{i}\right) = 2^{4n} {\rm det}\, \begin{pmatrix}
A_1 & A_2 \\\
-A_2 & A_1
\end{pmatrix} = 2^{4n} \left({\rm det}\, (\Phi''_{x\overline{x}}(y))\right)^2,
\endeq
which is a smooth strictly positive function near $y=x_0$. For future reference, let us also compute the quadratic form
\begeq
\label{eq2.64.1}
\left(\nabla^2_z f(y,0)\right)^{-1} \begin{pmatrix}
t \\\
s
\end{pmatrix} \cdot \begin{pmatrix}
t \\\
s
\end{pmatrix} = \frac{1}{2^2 i} \begin{pmatrix}
A_1 & A_2 \\\
-A_2 & A_1
\end{pmatrix}^{-1} \begin{pmatrix}
t \\\
s
\end{pmatrix} \cdot \begin{pmatrix}
t \\\
s
\end{pmatrix},
\endeq
dual to $\nabla^2_z f(y,0)$. To this end, a simple computation using (\ref{eq2.63.3}) shows that
\begeq
\label{eq2.64.2}
\begin{pmatrix}
A_1 & A_2 \\\
-A_2 & A_1
\end{pmatrix}^{-1} = \begin{pmatrix}
i & i \\\
-1 & 1
\end{pmatrix} \begin{pmatrix}
\left(\Phi''_{x\overline{x}}(y)\right)^{-1} & 0 \\\
0 & \left(\Phi''_{\overline{x}x}(y)\right)^{-1}
\end{pmatrix} \begin{pmatrix}
-i & -1 \\\
-i & 1
\end{pmatrix},
\endeq
and therefore we get
\begin{multline}
\label{eq2.64.3}
\left(\nabla^2_z f(y,0)\right)^{-1} \begin{pmatrix}
t \\\
s
\end{pmatrix} \cdot \begin{pmatrix}
t \\\
s
\end{pmatrix} = \frac{1}{2^2 i} \begin{pmatrix}
\left(\Phi''_{x\overline{x}}(y)\right)^{-1} & 0 \\\
0 & \left(\Phi''_{\overline{x}x}(y)\right)^{-1}
\end{pmatrix} \begin{pmatrix}
\overline{z}\\\
z
\end{pmatrix} \cdot \begin{pmatrix}
z \\\
\overline{z}
\end{pmatrix} \\
= \frac{1}{2i} \left(\Phi''_{x\overline{x}}(y)\right)^{-1} \overline{z}\cdot z.
\end{multline}
Here the quadratic form in (\ref{eq2.64.3}) can be regarded as the symbol of the second order constant coefficient differential operator on $\comp^n_z$ given by
\begeq
\label{eq2.64.4}
\frac{2}{i} \left(\Phi''_{x\overline{x}}(y)\right)^{-1} D_z \cdot D_{\overline{z}} = 2i \left(\Phi''_{x\overline{x}}(y)\right)^{-1} \partial_z \cdot \partial_{\overline{z}},
\endeq
where $\displaystyle D_z = \frac{1}{i} \partial_z = \frac{1}{2}(D_t - iD_s)$, $\displaystyle D_{\overline{z}} = \frac{1}{i} \partial_{\overline{z}} = \frac{1}{2}(D_t + iD_s)$.

\medskip
\noindent
It follows from (\ref{eq2.60}), (\ref{eq2.61}), and (\ref{eq2.64}) that we are in the position to apply complex stationary phase in the form given in~\cite[Theorem 7.7.5]{H_book} to derive a complete asymptotic expansion for $A_{\Gamma}a(y,\overline{y};h)$ given in (\ref{eq2.59}), as $h\rightarrow 0^+$. We obtain therefore that there exist differential operators $L_{j,y}(D)$ in $(t,s)$ of order $2j$, which are $C^{\infty}$ functions of $y\in {\rm neigh}(x_0,\comp^n)$, such that for each $N$ we have uniformly for $y\in {\rm neigh}(x_0,\comp^n)$,
\begeq
\label{eq2.65}
\left(A_{\Gamma}a\right)(y,\overline{y};h) = \sum_{j=0}^{N-1} h^j \left(L_{j,y}(D)b\right)(y,0) + {\cal O}_N(h^N).
\endeq
Let us also recall from~\cite[Theorem 7.7.5]{H_book}, using also (\ref{eq2.64.4}), the following explicit expressions for the operators $L_j$,
\begeq
\label{eq2.66}
(L_{j,y}(D)b)(y,0) = \frac{C_n \pi^n}{2^n {\rm det}\, (\Phi''_{x\overline{x}}(y))} \sum_{\nu - \mu = j} \sum_{2\nu \geq 3\mu} \frac{i^{\nu -j}}{\mu!\nu!} \left( \left(\Phi''_{x\overline{x}}(y)\right)^{-1}\partial_z \cdot \partial_{\overline{z}}\right)^{\nu} \left(g^{\mu}b\right)(y,0)
\endeq
where
$$
g(y,z) = f(y,z) - 2i \Phi''_{x\overline{x}}(y)\overline{z}\cdot z = {\cal O}(\abs{z}^3).
$$
In particular,
\begeq
\label{eq2.66.1}
L_{0,y} = \frac{C_n \pi^n}{2^n {\rm det}\, (\Phi''_{x\overline{x}}(y))}
\endeq
satisfies
\begeq
\label{eq2.67}
\frac{1}{C} \leq \abs{L_{0,y}} \leq C, \quad y\in {\rm neigh}(x_0,\comp^n).
\endeq
The expansion (\ref{eq2.65}) can be differentiated any number of times with respect to $y$, $\overline{y}$. Following~\cite[Section 3]{DelHiSj}, our purpose is now to show that there exists an amplitude $a(x,\widetilde{y};h)\in S^0_{{\rm cl}}({\rm neigh}((x_0,\overline{x_0}),\comp^{2n}))$ satisfying (\ref{eq2.29.1}), (\ref{eq2.30}) such that
\begeq
\label{eq2.67.1}
(A_{\Gamma}a)(y,\overline{y};h) = 1 + {\cal O}(h^{\infty}),
\endeq
for $y \in {\rm neigh}(x_0,\comp^n)$. Looking for $a$ in the form (\ref{eq2.29.1}), we may write in view of (\ref{eq2.65}),
\begeq
\label{eq2.67.2}
\left(A_{\Gamma}a\right)(y,\overline{y};h) \sim \sum_{\ell = 0}^{\infty} h^{\ell} c_{\ell}(y), \quad c_{\ell}(y) = \sum_{j+k= \ell} \left(L_{k,y}(D)b_j\right)(y,0),
\endeq
where $b_j(y,z) = a_j(y+z,\overline{y} - \overline{z})$. Using the expansion (\ref{eq2.67.2}), we shall determine successively $a_0,a_1,\ldots$ satisfying (\ref{eq2.30}), so that
\begeq
\label{eq2.67.3}
c_0(y) = L_{0,y}a_0(y,\overline{y}) = 1,
\endeq
\begeq
\label{eq2.67.4}
c_{\ell}(y) = \underset{j+k =\ell}{\sum} (L_{k,y}(D)b_j)(y,0) = 0,\quad \ell \geq 1.
\endeq
First, (\ref{eq2.67.3}) determines the $C^{\infty}$ function $a_0(y,\overline{y})$ uniquely, in view of (\ref{eq2.66.1}), (\ref{eq2.67}), and taking an almost holomorphic extension from the anti-diagonal, we obtain
$$
a_0(x,\widetilde{y}) \in C^{\infty}({\rm neigh}((x_0,\overline{x_0}),\comp^{2n})),
$$
satisfying (\ref{eq2.30}) for $j=0$. Assume next that $a_0,a_1,\ldots, a_{M-1}$, satisfying (\ref{eq2.30}), have been determined so that (\ref{eq2.67.4}) holds for $\ell \leq M-1$. To determine $a_M$, we consider the equation (\ref{eq2.67.4}) with $\ell = M$, writing
\begeq
\label{eq2.67.5}
c_{M}(y) = L_{0,y}a_{M}(y,\overline{y}) + \underset{\underset{j<M}{j+k =M}} {\sum} (L_{k,y}(D)b_j)(y,0) = 0.
\endeq
Here we may notice that the expression in the sum in (\ref{eq2.67.5}) only depends on the values of the $a_j$'s along the anti-diagonal, for $j\leq M-1$. Indeed, for each $\alpha \in \nat^n$, we have in view of the almost holomorphy of $a_j = a_j(x,\widetilde{y})$ along the anti-diagonal given in (\ref{eq2.30}), for $j\leq M-1$,
\begeq
\label{eq2.68}
(D^{\alpha}_z b_j)(y,0) = D^{\alpha}_z \left(a_j(y+z,\overline{y}-\overline{z})\right)|_{z=0} = \left(D^{\alpha}_x a_j\right)(y,\overline{y}) = D^{\alpha}_y \left(a_j(y,\overline{y})\right),
\endeq
\begeq
\label{eq2.69}
(D^{\alpha}_{\overline{z}} b_j)(y,0) = D^{\alpha}_{\overline{z}} \left(a_j(y+z,\overline{y}-\overline{z})\right)|_{z=0} = (-1)^{\alpha}\left(D^{\alpha}_{\widetilde{y}} a_j\right)(y,\overline{y}) = (-1)^{\alpha} D^{\alpha}_{\overline{y}} \left(a_j(y,\overline{y})\right).
\endeq
It follows that (\ref{eq2.67.5}) has a unique $C^{\infty}$ solution $a_M(y,\overline{y})$, for $y\in {\rm neigh}(x_0,\comp^n)$, and we may then take an almost holomorphic extension.

\medskip
\noindent
The discussion above may be summarized in the following theorem, which is the main result of this section.
\begin{theo}
\label{theo_amplitude}
There exists an elliptic symbol $a(x,\widetilde{y};h) \in S^0_{{\rm cl}}({\rm neigh}((x_0,\overline{x_0}),\comp^{2n}))$ of the form
\begeq
\label{eq2.70}
a(x,\widetilde{y};h) \sim \sum_{j=0}^{\infty} a_j(x,\widetilde{y}) h^j,
\endeq
in $C^{\infty}$, with $a_j \in C^{\infty}({\rm neigh}((x_0,\overline{x_0}),\comp^{2n}))$ satisfying
\begeq
\label{eq2.71}
\left(\partial_{\overline{x}}a_j\right)(x,\widetilde{y}) = {\cal O}(\abs{\widetilde{y}-\overline{x}}^{\infty}),\quad \left(\partial_{\overline{y}}a_j\right)(x,\widetilde{y}) =
{\cal O}(\abs{\widetilde{y}-\overline{x}}^{\infty}),\quad j=0,1,2,\ldots,
\endeq
such that we have
\begin{multline}
\label{eq2.72}
(A_{\Gamma}a)(y,\overline{y};h) = \frac{1}{h^n}\int\!\!\!\int_{\Gamma(y,\overline{y})} e^{\frac{2}{h} \varphi(y,\overline{y};x,\widetilde{y})} a(x,\widetilde{y};h)\, dx\, d\widetilde{y} \\
= 1 + {\cal O}(h^{\infty}), \quad y\in {\rm neigh}(x_0,\comp^n).
\end{multline}
Here $\Gamma(y,\overline{y})$ is a good contour for the function $(x,\widetilde{y}) \mapsto {\rm Re}\, \varphi(y,\overline{y};x,\widetilde{y})$.
The restrictions of the $a_j$'s to the anti-diagonal $\widetilde{y} = \overline{x}$ are uniquely determined, for $j\geq 0$, and we have
$$
a_0(x,\overline{x}) = A_n\, {\rm det}\, (\Phi''_{x\overline{x}}(x)), \quad x\in {\rm neigh}(x_0,\comp^n),
$$
with $A_n\neq 0$ depending on $n$ only.
\end{theo}

\section{Approximate reproducing property in the weak sense}
\label{sec:repr-property}
\setcounter{equation}{0}
\setcounter{equation}{0}
Let $V\Subset \Omega$ be a small open neighborhood of $x_0 \in \Omega$, with $C^{\infty}$--boundary. Let $\Psi \in C^{\infty}({\rm neigh}((x_0,\overline{x_0}),\comp^{2n}))$ be an almost holomorphic extension of the $C^{\infty}$ strictly plurisubharmonic weight function $\Phi$, so that (\ref{eq2.2}), (\ref{eq2.3}) hold. We may assume that $\Psi$, as well as the classical $C^{\infty}$ symbol $a$, introduced in Theorem  \ref{theo_amplitude} and satisfying (\ref{eq2.72}), are defined in a neighborhood of the closure of $V\times \rho(V)$. Here $\rho(x) = \overline{x}$ is the map of complex conjugation.

\medskip
\noindent
Let us set
\begeq
\label{eq3.1}
\widetilde{\Pi}_{V} u(x) = \frac{1}{h^n} \int_{V} e^{\frac{2}{h}\Psi(x,\overline{y})} a(x,\overline{y};h) u(y) e^{-\frac{2}{h}\Phi(y)}\, dy\, d\overline{y}, \quad u\in L^2_{\Phi}(V):= L^2(V,e^{-2\Phi/h}L(dx)).
\endeq
It follows from Proposition \ref{prop1} and the Schur test that
\begeq
\label{eq3.3}
\widetilde{\Pi}_{V} = {\cal O}(1): L^2_{\Phi}(V) \rightarrow L^2_{\Phi}(V).
\endeq
Furthermore, combining (\ref{eq2.3}), (\ref{eq2.70}), (\ref{eq2.71}) with the Schur test, we obtain
\begeq
\label{eq3.3.0.1}
\overline{\partial} \circ \widetilde{\Pi}_{V} = {\cal O}(h^{\infty}): L^2_{\Phi}(V) \rightarrow L^2_{\Phi,(0,1)}(V).
\endeq
Here the target space is a space of $(0,1)$--forms on $V$. Letting $u\in L^2_{\Phi}(V)$ be holomorphic, we can express $\widetilde{\Pi}_V u$ in the polarized form, as a contour integral in $\comp^{2n}_{y,\widetilde{y}}$ of a $(2n,0)$--form,
\begeq
\label{eq3.3.0.2}
\widetilde{\Pi}_{V} u(x) = \frac{1}{h^n} \int\!\!\!\int_{\Gamma_V} e^{\frac{2}{h}(\Psi(x,\widetilde{y}) - \Psi(y,\widetilde{y}))} a(x,\widetilde{y};h) u(y)\, dy\, d\widetilde{y}, \quad u\in H_{\Phi}(V):={\rm Hol}(V)\cap L^2_{\Phi}(V).
\endeq
Here the contour of integration $\Gamma_V \subset V \times \rho(V)$ is given by
\begeq
\label{eq3.2}
\Gamma_V = \{\widetilde{y}=\overline{y},\,\,y\in V\}.
\endeq

\medskip
\noindent
The purpose of this section is to show that the operator $\widetilde{\Pi}_V$ in (\ref{eq3.3.0.2}) satisfies an approximate reproducing property, in the weak formulation. Specifically, we shall prove that for a suitable class of $u,v\in H_{\Phi}(V)$, the sesquilinear form
\begeq
\label{eq3.3.1}
H_{\Phi}(V) \times H_{\Phi}(V) \ni (u,v) \mapsto (\widetilde{\Pi}_{V} u,v)_{L^2_{\Phi}(V)}
\endeq
agrees, up to an ${\cal O}(h^{\infty})$--error, with the scalar product $(u,v)_{H_{\Phi}(V)}$. In~\cite{DelHiSj}, we have observed that this result cannot be expected to hold if $u,v$ are general elements of $H_{\Phi}(V)$, and similar to~\cite{DelHiSj}, we shall demand that $v$ should belong to an exponentially weighted space of holomorphic functions of the form $H_{\Phi_1}(V)$, where $\Phi_1 \leq \Phi$, with strict inequality away from a small neighborhood of $x_0$. The following theorem is the main result of this section.

\begin{theo}
\label{reproducing_scalar_product}
There exists a small open neighborhood $W \Subset V$ of $x_0$ with $C^{\infty}$--boundary such that for every $\Phi_1 \in C(\Omega; \real)$,
$\Phi_1 \leq \Phi$, with $\Phi_1 < \Phi$ on $\Omega \backslash \overline{W}$, and every $N\in {\bf N}$ there exists $C_N$ such that for all $u\in H_{\Phi}(V)$, $v\in H_{\Phi_1}(V)$, we have
\begeq
\label{eq3.3.2}
\abs{(\widetilde{\Pi}_V u,v)_{L^2_{\Phi}(V)} - (u,v)_{H_{\Phi}(V)}} \leq C_N h^N \norm{u}_{H_{\Phi}(V)} \norm{v}_{H_{\Phi_1}(V)}.
\endeq
\end{theo}

\medskip
\noindent
Following~\cite[Section 4]{DelHiSj}, the proof of Theorem \ref{reproducing_scalar_product} will proceed by a contour deformation argument. Compared with the analytic case treated in~\cite{DelHiSj}, here, when justifying the contour deformation, we shall have to take into account the lack of holomorphy in the integrand in (\ref{eq3.3.0.2}), giving rise to an additional correction term, to be estimated.

\medskip
\noindent
Let $W \Subset V_1 \Subset V_2  \Subset V$ be open neighborhoods of $x_0$ with $C^{\infty}$--boundaries, with $W$ to be chosen small enough, and let $\Phi_1 \in C(\Omega; \real)$ be such that
\begeq
\label{eq3.4}
\Phi_1 \leq \Phi \,\, \wrtext{in}\,\, \Omega, \quad \Phi_1 < \Phi \,\, \wrtext{on}\,\, \Omega\backslash \overline{W}.
\endeq
Arguing as in~\cite[Section 4]{DelHiSj}, we find that the scalar product
\begeq
\label{eq3.5}
(\widetilde{\Pi}_{V} u,v)_{L^2_{\Phi}(V)} = \int_{V} \widetilde{\Pi}_{V} u(x) \overline{v(x)} e^{-2\Phi(x)/h}\, L(dx), \quad u\in H_{\Phi}(V),\,\, v\in H_{\Phi_1}(V),
\endeq
takes the form
\begeq
\label{eq3.6}
(\widetilde{\Pi}_{V} u,v)_{L^2_{\Phi}(V)} = \int_{V_1} \widetilde{\Pi}_{V_2} u(x) \overline{v(x)} e^{-2\Phi(x)/h}\, L(dx) +
{\cal O}(1) e^{-\frac{1}{C h}} \norm{u}_{H_{\Phi}(V)} \norm{v}_{H_{\Phi_1}(V)}.
\endeq
Here and below $C>0$ is independent of $u$, $v$, and similar to (\ref{eq3.3.0.2}), we have written
\begeq
\label{eq3.6.3}
\widetilde{\Pi}_{V_2} u(x) = \frac{1}{h^n} \int\!\!\!\int_{\Gamma_{V_2}} e^{\frac{2}{h}(\Psi(x,\widetilde{y}) - \Psi(y,\widetilde{y}))} a(x,\widetilde{y};h) u(y)\, dy\, d\widetilde{y}.
\endeq
Next, an application of~\cite[Proposition 2.2]{DelHiSj} gives that there exists $\eta > 0$ such that
\begeq
\label{eq3.6.4}
v(x) = \int_V v_z(x) dz\,d\overline{z} + {\cal O}(1) \norm{v}_{H_{\Phi}(V)} e^{\frac{1}{h}(\Phi(x) - \eta)}, \quad x\in V_1.
\endeq
Here
\begeq
\label{eq3.6.5}
v_z(x) = \frac{1}{(2 \pi h)^n} e^{\frac{i}{h}(x-z)\cdot \theta(x,z)} v(z)\chi(z) {\rm det}\,(\partial_{\overline{z}}\theta(x,z)) \in {\rm Hol}(V),
\endeq
and $\theta(x,z)$ depends holomorphically on $x\in V$ with
\begeq
\label{eq3.6.6}
-{\rm Im}\, \left((x-z)\cdot \theta(x,z)\right) + \Phi(z) \leq \Phi(x) - \delta \abs{x-z}^2, \quad x,z\in V,
\endeq
for some $\delta >0$. The function $\chi\in C^{\infty}_0(V;[0,1])$ in (\ref{eq3.6.5}) satisfies $\chi = 1$ in $V_2$.

\medskip
\noindent
The resolution of the identity (\ref{eq3.6.4}), (\ref{eq3.6.5}), (\ref{eq3.6.6}) is valid for an arbitrary element of $H_{\Phi}(V)$, and restricting the attention to $v\in H_{\Phi_1}(V)$, we get, letting $W \Subset W_1 \Subset V_1$,
\begeq
\label{eq3.7}
v(x) = \int_{W_1} v_z(x)\, dz\,d\overline{z} + {\cal O}(1) \norm{v}_{H_{\Phi_1}(V)} e^{\frac{1}{h}(\Phi(x) - \frac{1}{C})}, \quad x\in V_1.
\endeq
We get, combining (\ref{eq3.6}), (\ref{eq3.7}), and (\ref{eq3.3}),
\begin{multline}
\label{eq3.9}
(\widetilde{\Pi}_{V} u,v)_{L^2_{\Phi}(V)} \\ = \int_{W_1}\!\int_{V_1} \widetilde{\Pi}_{V_2} u(x) \overline{v_z(x)} e^{-2\Phi(x)/h}\, L(dx)\, dz\, d\overline{z} + {\cal O}(1) e^{-\frac{1}{C h}}
\norm{u}_{H_{\Phi}(V)} \norm{v}_{H_{\Phi_1}(V)} \\
= \int_{W_1} (\widetilde{\Pi}_{V_2} u,v_z)_{L^2_{\Phi}(V_1)} dz\, d\overline{z} + {\cal O}(1) e^{-\frac{1}{C h}} \norm{u}_{H_{\Phi}(V)} \norm{v}_{H_{\Phi_1}(V)}.
\end{multline}
Similar to~\cite{DelHiSj}, the advantage of the representation (\ref{eq3.9}) lies in the good localization properties of the holomorphic functions $v_z$, for $z\in W_1$, in view of (\ref{eq3.6.6}).

\bigskip
\noindent
The following key observation is analogous to~\cite[Proposition 4.2]{DelHiSj}.
\begin{prop}
\label{good_contours}
Let $\delta >0$ be small and let us set for $z\in V$, $(x,\widetilde{x},y,\widetilde{y}) \in V\times \rho(V)\times V \times \rho(V) \subset \comp^{4n}$,
\begin{multline}
\label{eq3.12}
G_z(x,\widetilde{x},y,\widetilde{y}) = 2{\rm Re}\, \Psi(x,\widetilde{y}) - 2{\rm Re}\, \Psi(y,\widetilde{y}) + \Phi(y) + F_{z}(\widetilde{x}) - 2{\rm Re}\, \Psi(x,\widetilde{x}) \\
= 2{\rm Re}\, \varphi(y,\widetilde{x}; x,\widetilde{y}) - 2{\rm Re}\, \Psi(y,\widetilde{x}) + \Phi(y) + F_{z}(\widetilde{x}),
\end{multline}
where
\begeq
\label{eq3.12.1}
F_{z}(\widetilde{x}) = \Phi(\overline{\widetilde{x}}) - \delta \abs{\widetilde{x} - \overline{z}}^2.
\endeq
The $C^{\infty}$ function $G_{z}$ has a non-degenerate critical point at $(z,\overline{z},z,\overline{z})$ of signature $(4n,4n)$, with the critical value $0$. The following submanifolds of $\comp^{4n}$ are good contours for $G_{z}$ in a neighbourhood of $(z,\overline{z},z,\overline{z})$, i.e. they are both of dimension $4n$, pass through the critical point, and are such that the Hessian of $G_z$ along the contours is negative definite:
\begin{enumerate}
\item The product contour
\begeq
\label{eq3.12.2}
\Gamma_{V}\times \Gamma_{V}=\{(x,\widetilde{x},y,\widetilde{y});\,\widetilde{x} = \overline{x},\,\,\widetilde{y} = \overline{y},\,\, x\in V,\,\,y\in V\}.
\endeq
\item The composed contour
\begeq
\label{eq3.12.3}
\{(x,\widetilde{x},y,\widetilde{y});\, (y,\widetilde{x}) \in \Gamma_{V},\, (x,\widetilde{y}) \in \Gamma(y, \widetilde{x})\}.
\endeq
Here $\Gamma(y,\widetilde{x}) \subset \comp^{2n}_{x,\widetilde{y}}$ is a good contour for the $C^{\infty}$ function $(x,\widetilde{y}) \mapsto {\rm Re}\,\varphi(y,\widetilde{x};x,\widetilde{y})$ described in {\rm (\ref{eq2.28})}, see also Proposition {\rm \ref{prop2}}.
\end{enumerate}
\end{prop}

\begin{proof}
We have the Taylor expansions at $(x,\widetilde{x},y,\widetilde{y}) = (z,\overline{z},z,\overline{z})\in \comp^{4n}$,
\begin{multline}
\label{T1}
2{\rm Re}\,\Psi(x,\widetilde{y}) - 2{\rm Re}\,\Psi(y,\widetilde{y})= 2\Phi(z) + 2{\rm Re}\,\left(\Phi'_x(z)\cdot (x-z) + \Phi'_{\overline{x}}(z)\cdot (\widetilde{y} - \overline{z})\right) \\
- 2\Phi(z) - 2{\rm Re}\,\left(\Phi'_x(z)\cdot (y-z) + \Phi'_{\overline{x}}(z)\cdot (\widetilde{y} - \overline{z})\right) + {\cal O}((x-z,\widetilde{y}-\overline{z},y-z)^2)
\\ = 2{\rm Re}\,\left(\Phi'_x(z)\cdot (x-z) - \Phi'_x(z)\cdot (y-z)\right) + {\cal O}((x-z,\widetilde{y}-\overline{z},y-z)^2),
\end{multline}
\begeq
\label{T2}
\Phi(y) + F_z(\widetilde{x}) = 2\Phi(z) + 2{\rm Re}\,\left(\Phi'_x(z)\cdot (y-z) + \Phi'_x(z)\cdot (\overline{\widetilde{x}}-z)\right) + {\cal O}((y-z,\widetilde{x}-\overline{z})^2),
\endeq
\begeq
\label{T3}
2{\rm Re}\,\Psi(x,\widetilde{x}) = 2\Phi(z) + 2{\rm Re}\, \left(\Phi'_x(z)\cdot (x-z) + \Phi'_{\overline{x}}(z)\cdot (\widetilde{x} - \overline{z})\right) + {\cal O}((x-z,\widetilde{x}-\overline{z})^2).
\endeq
Here we have also used the almost holomorphy of $\Psi$. We get, combining (\ref{T1}), (\ref{T2}), (\ref{T3}) and using (\ref{eq3.12}),
\begeq
\label{T4}
G_z(x,\widetilde{x},y,\widetilde{y}) = {\cal O}\left({\rm dist}\left((x,\widetilde{x},y,\widetilde{y}),(z,\overline{z},z,\overline{z})\right)^2\right).
\endeq
The point $(z,\overline{z},z,\overline{z})$ is therefore a critical point of $G_z$ with the critical value $0$. When showing that it is non-degenerate of signature $(4n,4n)$, we observe that in view of Proposition \ref{prop1}, we have, using the expression for $G_z$ on the first line of (\ref{eq3.12}),
\begeq
\label{eq3.13}
G_{z}(x,\overline{x},y,\overline{y}) \leq -\frac{1}{C}\abs{y-x}^2 - \delta\abs{x-z}^2 \leq -\frac{1}{C}\abs{x-z}^2 - \frac{1}{C}\abs{y-z}^2.
\endeq
The contour (\ref{eq3.12.2}) is therefore good for $G_z$, and using also the fact that the quadratic part of the Taylor expansion of $G_z$ at the point $(z,\overline{z},z,\overline{z})$ is a plurisubharmonic quadratic form on $\comp^{4n}$, we conclude that $(z,\overline{z},z,\overline{z})$ is a non-degenerate critical point of $G_z$, of signature $(4n,4n)$. The verification of the fact that the contour (\ref{eq3.12.3}) is also good for $G_z$ is performed exactly as in the proof of Proposition 4.2 of~\cite{DelHiSj}, using the expression for $G_z$ given on the second line of (\ref{eq3.12}). The proof is complete.
\end{proof}

\bigskip
\noindent
We can now carry out the contour deformation argument for the scalar product $(\widetilde{\Pi}_{V_2}u,v_z)$, alluded to above.

\begin{prop}\label{prop:repr-F_z}
Let $v_z$ be of the form {\rm (\ref{eq3.6.5})}, {\rm (\ref{eq3.6.6})}. There exists an open neighborhood $W_1\Subset V_1$ of $x_0$ such that for all $u\in H_{\Phi}(V)$, we have,
\begeq
\label{eq3.17}
(\widetilde{\Pi}_{V_2} u, v_z)_{L^2_{\Phi}(V_1)} = (u,v_z)_{H_{\Phi}(V_1)} + {\cal O}(h^{\infty}) \norm{u}_{H_{\Phi}(V)} \abs{v(z)} e^{-\Phi(z)/h},
\endeq
uniformly in $z\in W_1$.
\end{prop}
\begin{proof}
Writing the Lebesgue measure on $\comp^n$ in the form $L(dx) = C_n dx\, d\overline{x}$, let us express the scalar product in the space $H_{\Phi}(V_1)$ in the polarized form,
\begeq
\label{eq3.18}
(f,g)_{H_{\Phi}(V_1)} = \int_{V_1} f(x) \overline{g(x)} e^{-\frac{2\Phi(x)}{h}}\, L(dx) = C_n  \int\!\!\!\int_{\Gamma_{V_1}} f(x) g^*(\widetilde{x}) e^{-\frac{2}{h}\Psi(x,\widetilde{x})}\, dx\, d\widetilde{x}.
\endeq
Here the contour $\Gamma_{V_1}$ is defined as in (\ref{eq3.2}) and we have also set
\begeq
\label{eq3.19}
g^*(\widetilde{x}) = \overline{g(\overline{\widetilde{x}})} \in H_{\widehat{\Phi}}(\rho(V_1)), \quad \widehat{\Phi}(\widetilde{x}) = \Phi(\overline{\widetilde{x}}).
\endeq
In view of (\ref{eq3.6.3}) and (\ref{eq3.18}), we may write 
\begin{multline}
\label{eq3.20}
(\widetilde{\Pi}_{V_2} u, v_z)_{L^2_{\Phi}(V_1)} \\
= \frac{C_n}{h^n} \int\!\!\!\int_{\Gamma_{V_1}} \left(\int\!\!\!\int_{\Gamma_{V_2}} e^{\frac{2}{h}(\Psi(x,\widetilde{y}) - \Psi(y,\widetilde{y}))} a(x,\widetilde{y};h) u(y)\, dy\, d\widetilde{y}\right) v_z^*(\widetilde{x}) e^{-\frac{2}{h}\Psi(x,\widetilde{x})}\, dx\,d\widetilde{x} \\
= \int\!\!\!\int\!\!\!\int\!\!\!\int_{\Gamma_{V_1}\times \Gamma_{V_2}} \omega.
\end{multline}
Here $\omega$ is the $(4n,0)$--differential form on $\comp^{4n}$ of the form
\begeq
\label{eq3.20.1}
\omega = f(x,\widetilde{x},y,\widetilde{y})\, dx\wedge \,d\widetilde{x}\wedge \,dy\wedge \,d\widetilde{y},
\endeq
where $f\in C^{\infty}(V\times \rho(V) \times V \times \rho(V))$ is given by
\begeq
\label{eq3.20.2}
f(x,\widetilde{x},y,\widetilde{y}) = \frac{C_n}{h^n} e^{\frac{2}{h}(\Psi(x,\widetilde{y}) - \Psi(y,\widetilde{y}))}  a(x,\widetilde{y};h) u(y) v_z^*(\widetilde{x}) e^{-\frac{2}{h}\Psi(x,\widetilde{x})}.
\endeq
When estimating $f$, we notice first, in view of (\ref{eq3.6.5}), (\ref{eq3.6.6}),
\begeq
\label{eq3.21}
\abs{v_z^*(\widetilde{x})} \leq \frac{{\cal O}(1)}{h^n} \abs{v(z)} e^{-\Phi(z)/h} e^{F_z(\widetilde{x})/h}, \quad \widetilde{x} \in \rho(V_1),
\endeq
where $F_z$ is the strictly plurisubharmonic function in $\rho(V_1)$ given in (\ref{eq3.12.1}). Combining (\ref{eq3.21}) with~\cite[Proposition 2.3]{DelHiSj}, we get
\begeq
\label{eq3.23}
\abs{f(x,\widetilde{x},y,\widetilde{y})} \leq \frac{{\cal O}(1)}{h^{3n}} \norm{u}_{H_{\Phi}(V)} \abs{v(z)} e^{-\Phi(z)/h} e^{G_z(x,\widetilde{x}, y,\widetilde{y})/h},
\endeq
for $(x,\widetilde{x}, y,\widetilde{y}) \in V_1 \times \rho(V_1) \times V_2 \times \rho(V_2)$, with $G_z(x,\widetilde{x}, y,\widetilde{y})$ given in (\ref{eq3.12}). Proposition \ref{good_contours} tells us that the contour $\Gamma_1:= \Gamma_{V_1}\times \Gamma_{V_2}$ and the composed contour $\Gamma_2$ defined in (\ref{eq3.12.3}) are both good for $G_z$, and as reviewed in the proof of Proposition \ref{prop_contour}, there exists therefore a $C^{\infty}$ homotopy between the contours $\Gamma_1$, $\Gamma_2$, well defined for all $z$ in a small neighborhood of $x_0$, passing through good contours only, uniformly for $z$ close enough to $x_0$. Let $\Sigma \subset \comp^{4n}$ be the $(4n+1)$--dimensional contour of integration naturally associated to the homotopy above. We have, with a suitable orientation,
\begeq
\label{eq3.23.1}
\partial \Sigma - (\Gamma_1 - \Gamma_2) \subset \left\{(x,\widetilde{x},y,\widetilde{y});\, G_z(x,\widetilde{x},y,\widetilde{y}) \leq
-\frac{1}{{\cal O}(1)}\right\},
\endeq
uniformly for all $z$ in a small neighborhood of $x_0$. We may write therefore, using Stokes' formula, (\ref{eq3.23}), and (\ref{eq3.23.1}), for all $z$ in a small neighborhood of $x_0$,
\begin{multline}
\label{eq3.24}
\int\!\!\!\int\!\!\!\int\!\!\!\int_{\Gamma_1} \omega - \int\!\!\!\int\!\!\!\int\!\!\!\int_{\Gamma_2} \omega = \int\!\!\!\int\!\!\!\int\!\!\!\int\!\!\!\int_{\Sigma} d\omega + {\cal O}(1) \norm{u}_{H_{\Phi}(V)} \abs{v(z)} e^{-\Phi(z)/h} e^{-1/Ch} \\
= \int\!\!\!\int\!\!\!\int\!\!\!\int\!\!\!\int_{\Sigma} \overline{\partial} f \wedge dx\wedge d\widetilde{x}\wedge dy \wedge d\widetilde{y} +
{\cal O}(1) \norm{u}_{H_{\Phi}(V)} \abs{v(z)} e^{-\Phi(z)/h} e^{-1/Ch}.
\end{multline}
Using (\ref{eq3.20.2}), we compute
\begin{multline}
\label{eq3.24.1}
\partial_{\overline{x}}f(x,\widetilde{x},y,\widetilde{y}) \\
= \frac{C_n}{h^n} e^{\frac{2}{h}(\Psi(x,\widetilde{y}) - \Psi(y,\widetilde{y}) - \Psi(x,\widetilde{x}))} u(y) v_z^*(\widetilde{x})
\left(\frac{2}{h} \left(\partial_{\overline{x}}\Psi(x,\widetilde{y}) - \partial_{\overline{x}}\Psi(x,\widetilde{x})\right)a
+ \partial_{\overline{x}} a\right),
\end{multline}
and recalling (\ref{eq2.3}), (\ref{eq2.71}), we infer for all $N\in \nat$,
\begin{multline}
\label{eq3.24.2}
\abs{\partial_{\overline{x}}f} \leq \frac{{\cal O}_N(1)}{h^{3n+1}} \norm{u}_{H_{\Phi}(V)} \abs{v(z)} e^{-\Phi(z)/h} e^{G_z(x,\widetilde{x}, y,\widetilde{y})/h}
\left(\abs{\overline{x}-\widetilde{y}}^N + \abs{\overline{x}-\widetilde{x}}^N + h^N\right)\\
\leq \frac{{\cal O}_N(1)}{h^{3n+1}} \norm{u}_{H_{\Phi}(V)} \abs{v(z)} e^{-\frac{\Phi(z)}{h}}
e^{-{\rm dist}\left((x,\widetilde{x},y,\widetilde{y}),(z,\overline{z},z,\overline{z})\right)^2/Ch} \left(\abs{\overline{x}-\widetilde{y}}^N + \abs{\overline{x}-\widetilde{x}}^N + h^N\right).
\end{multline}
Here we write
\begeq
\label{eq3.24.3}
\abs{\overline{x} - \widetilde{y}}^N \leq \left(\abs{x-z}+ \abs{\widetilde{y} - \overline{z}}\right)^N
\leq {\cal O}_N(1)\left({\rm dist}\left((x,\widetilde{x},y,\widetilde{y}),(z,\overline{z},z,\overline{z}\right)\right)^N,
\endeq
and similarly,
\begeq
\label{eq3.24.4}
\abs{\overline{x} - \widetilde{x}}^N \leq {\cal O}_N(1)\left({\rm dist}\left((x,\widetilde{x},y,\widetilde{y}),(z,\overline{z},z,\overline{z}\right)\right)^N.
\endeq
We get therefore, combining (\ref{eq3.24.2}), (\ref{eq3.24.3}), (\ref{eq3.24.4}),
\begeq
\label{eq3.24.5}
\abs{\partial_{\overline{x}}f(x,\widetilde{x},y,\widetilde{y})} \leq {\cal O}(h^{\infty}) \norm{u}_{H_{\Phi}(V)} \abs{v(z)} e^{-\frac{\Phi(z)}{h}}.
\endeq
Similar computations and estimates show that the bound (\ref{eq3.24.5}) is also valid for $\partial_{\overline{\widetilde{x}}}f$, $\partial_{\overline{y}}f$, and  $\partial_{\overline{\widetilde{y}}}f$. We conclude that there exists a small open neighborhood $W_1 \Subset V_1$ of $x_0$ such that for all $z\in W_1$, the right hand side of (\ref{eq3.24}) is of the form
\begeq
\label{eq3.24.6}
{\cal O}(h^{\infty}) \norm{u}_{H_{\Phi}(V)} \abs{v(z)} e^{-\frac{\Phi(z)}{h}}.
\endeq
It follows therefore from (\ref{eq3.20}), (\ref{eq3.24}), and (\ref{eq3.24.6}) that for all $z\in W_1$, the scalar product $(\widetilde{\Pi}_{V_2}u,v_z)_{L^2_{\Phi}(V_1)}$ is equal to
\begin{multline}
\label{eq3.25}
C_n \int\!\!\!\int_{\Gamma_{V_1}} \left(\frac{1}{h^n} \int\!\!\!\int_{\Gamma(y,\widetilde{x})\cap (V_1\times \rho(V_1))} e^{\frac{2}{h} \varphi(y,\widetilde{x}; x,\widetilde{y})} a(x,\widetilde{y};h) dx\, d\widetilde{y}\right) u(y) v_{z}^*(\widetilde{x}) e^{-\frac{2}{h} \Psi(y,\widetilde{x})} dy\, d\widetilde{x} \\
+ {\cal O}(h^{\infty}) \norm{u}_{H_{\Phi}(V)} \abs{v(z)} e^{-\frac{\Phi(z)}{h}}.
\end{multline}
Here in the contour of integration in the inner integral we have $(y,\widetilde{x})\in \Gamma_{V_1} \Longleftrightarrow \widetilde{x} = \overline{y}$, $y\in V_1$, and by Theorem \ref{theo_amplitude} we obtain therefore that the inner integral is equal to $1 + {\cal O}(h^{\infty})$, provided that the neighborhood $V_1$ is small enough. The integral (\ref{eq3.25}) is therefore equal to
\begeq
\label{eq3.27}
(u,v_z)_{H_{\Phi}(V_1)} + {\cal O}(h^{\infty}) \norm{u}_{H_{\Phi}(V)} \abs{v(z)} e^{-\frac{\Phi(z)}{h}},
\endeq
uniformly for $z\in W_1$. The proof is complete.
\end{proof}

\bigskip
\noindent
We can now finish the proof of Theorem \ref{reproducing_scalar_product} as in~\cite[Section 4]{DelHiSj}, letting $W\Subset W_1$, where $W_1$ is as in Proposition \ref{prop:repr-F_z}. We get, in view of (\ref{eq3.9}) and (\ref{eq3.17}),
\begeq
\label{eq3.28}
(\widetilde{\Pi}_{V} u, v)_{L^2_{\Phi}(V)} = \int_{W_1} (u,v_z)_{H_{\Phi}(V_1)}\, dz\,d\overline{z} + {\cal O}(h^{\infty}) \norm{u}_{H_{\Phi}(V)} \norm{v}_{H_{\Phi_1}(V)}.
\endeq
Using (\ref{eq3.7}), we can also write
\begeq
\label{eq3.29}
(u, v)_{H_{\Phi}(V)} = \int_{W_1} (u,v_z)_{H_{\Phi}(V_1)}\, dz\,d\overline{z} + {\cal O}(1)e^{-\frac{1}{C h}} \norm{u}_{H_{\Phi}(V)} \norm{v}_{H_{\Phi_1}(V)}.
\endeq
The proof of Theorem \ref{reproducing_scalar_product} is complete.

\section{Completing the proof of Theorem \ref{Theorem1}}
\label{sec:end-proof}
\setcounter{equation}{0}
Let us first pass from the scalar products in Theorem \ref{reproducing_scalar_product} to weighted $L^2$ norm estimates. This will be done similarly to~\cite[Section 5]{DelHiSj}, with appropriate modifications to accommodate the fact that the operator $\widetilde{\Pi}_V$ in (\ref{eq3.1}) does not quite produce  holomorphic functions. To this end, let $0\leq \chi_1\in C^{\infty}(\Omega;\real)$ be such that $\chi_1 > 0$ on $\Omega\backslash\overline{W}$, where $W\Subset V$ is as in Theorem \ref{reproducing_scalar_product}, and let us set
\begeq
\label{eq4.1}
\Phi_1(x) = \Phi(x) - \delta \raisebox{.1 em}{$\chi_1(x)$},
\endeq
for $\delta >0 $ small enough. In particular, $\Phi_1$ is strictly plurisubharmonic in $V$,
\begeq
\label{eq4.2}
\sum_{j,k=1}^n \frac{\partial^2 \Phi_1}{\partial x_j \partial \overline{x}_k}(x) \xi_j \overline{\xi}_k \geq \frac{\abs{\xi}^2}{{\cal O}(1)}, \quad x\in V, \quad \xi \in \comp^n.
\endeq
In what follows, without loss of generality, we shall assume that the bounded open set $V$ is convex, and we may even take it to be an open ball centered at $x_0$.

\medskip
\noindent
Using Proposition \ref{prop1} together with the Schur test, we obtain
\begeq
\label{eq4.3}
\widetilde{\Pi}_V = {\cal O}(1): L^2_{\Phi_1}(V) \rightarrow L^2_{\Phi_2}(V).
\endeq
Here $\Phi_2 = \Phi - \chi_2$, where $0\leq \chi_2 \in C(\Omega;\real)$ is given by
\begeq
\label{eq4.3.1}
\chi_2(x) = \inf_{y\in V_0} \left(\frac{\abs{x-y}^2}{2C} + \delta \chi_1(y)\right),
\endeq
with $C>0$ and $V_0$ being an open ball centered at $x_0$ such that $V \Subset V_0 \Subset \Omega$. In particular, we see that
\begeq
\label{eq4.4}
\Phi_2 \leq \Phi \,\, \wrtext{in}\,\, \Omega, \quad \Phi_2 < \Phi \,\, \wrtext{on}\,\, \Omega\setminus \overline{W},
\endeq
and
\begeq
\label{eq4.5}
\Phi_1 \leq \Phi_2 \,\, \wrtext{in}\,\, V.
\endeq
Let us now take a closer look at the function $\chi_2$ in (\ref{eq4.3.1}). When doing so, let us observe first that the infimum in (\ref{eq4.3.1}) is attained at a unique point in $\overline{V_0}$, in view of the strict convexity of the function
$\displaystyle \overline{V_0} \ni y \mapsto F_x(y) := \frac{\abs{x-y}^2}{2C} + \delta \chi_1(y)$, for $\delta > 0$ small enough. On the other hand, when $x\in V$, the function $F_x(y)$ has a unique critical point $y_c(x)$ in $V_0$, for $\delta > 0$ small enough, which is a non-degenerate local minimum. Indeed, we have
$$
F'_x(y) = 0 \Leftrightarrow CF'_0(y) = x,
$$
and the map $V_0 \ni y \mapsto CF'_0(y) = y + C\delta \chi'_1(y)$ is a $C^{\infty}$ diffeomorphism from $V_0$ onto its image, which contains the open ball $V$ as a relatively compact subset, for $\delta > 0$ small enough. We have
\begeq
\label{eq4.6}
y_c(x) = x  -C\delta \chi'_1(x) + {\cal O}(\delta^2), \quad x\in V,
\endeq
and we conclude therefore that
\begeq
\label{eq4.6.1}
\chi_2(x) = F_x(y_c(x)) = \delta \chi_1(x) - \frac{C}{2}\delta^2 \abs{\chi'_1(x)}^2 + {\cal O}(\delta^3),\quad x\in V.
\endeq
In particular, we have $\norm{\Phi - \Phi_2}_{C^2(\overline{V})} = {\cal O}(\delta)$ is small enough, so that the function $\Phi_2$ is strictly plurisubharmonic in $V$,
\begeq
\label{eq4.7}
\sum_{j,k=1}^n \frac{\partial^2 \Phi_2}{\partial x_j \partial \overline{x}_k}(x) \xi_j \overline{\xi}_k \geq \frac{\abs{\xi}^2}{{\cal O}(1)}, \quad x\in V, \quad \xi \in \comp^n.
\endeq

\medskip
\noindent
Next, we let
\begeq
\label{eq4.8}
\Pi_{\Phi_2}: L^2_{\Phi_2}(V) \rightarrow H_{\Phi_2}(V)
\endeq
be the orthogonal projection. We shall make use of the following observation.
\begin{prop}
\label{prop_orth_proj}
We have
\begeq
\label{eq4.8.1}
\Pi_{\Phi_2} \widetilde{\Pi}_V - \widetilde{\Pi}_V = {\cal O}(h^{\infty}): L^2_{\Phi_1}(V) \rightarrow L^2_{\Phi_2}(V).
\endeq
\end{prop}
\begin{proof}
Using (\ref{eq2.3}), (\ref{eq2.29.1}), (\ref{eq2.30}), and Proposition \ref{prop1} together with the Schur test, we observe that
\begeq
\label{eq4.8.2}
\overline{\partial} \circ \widetilde{\Pi}_V = {\cal O}(h^{\infty}): L^2_{\Phi_1}(V) \rightarrow L^2_{\Phi_2,(0,1)}(V),
\endeq
where the target space in (\ref{eq4.8.2}) is a space of $(0,1)$--forms. Given $f\in L^2_{\Phi_1}(V)$, the solution of the equation
\begeq
\label{eq4.9}
\overline{\partial}u = \overline{\partial} \widetilde{\Pi}_V f
\endeq
of the minimal $L^2_{\Phi_2}(V)$--norm is given by $(1-\Pi_{\Phi_2})\widetilde{\Pi}_{V}f$, and an application of H\"ormander's $L^2$--estimates for the $\overline{\partial}$--operator in the open convex set $V$ and the strictly plurisubharmonic weight $\Phi_2$, \cite[Proposition 4.2.5]{Horm_Conv}, gives that
\begeq
\label{eq4.91}
\norm{(1-\Pi_{\Phi_2})\widetilde{\Pi}_V f}_{L^2_{\Phi_2}(V)} \leq {\cal O}(h^{1/2}) \norm{\overline{\partial} \widetilde{\Pi}_V f}_{L^2_{\Phi_2}(V)} \leq {\cal O}(h^{\infty}) \norm{f}_{L^2_{\Phi_1}(V)}.
\endeq
Here we have also used (\ref{eq4.8.2}). The proof is complete.
\end{proof}

\bigskip
\noindent
It is now easy to derive a suitable estimate for the operator $\widetilde{\Pi}_V -1$, proceeding as in~\cite{DelHiSj}. Let
$u\in H_{\Phi_1}(V)$, where $\Phi_1 \in C^{\infty}(\Omega;\real)$ is given by (\ref{eq4.1}), and let us apply Theorem \ref{reproducing_scalar_product}, with
\begeq
\label{eq4.92}
v = \Pi_{\Phi_2}\left((\widetilde{\Pi}_V -1)u\right) = \Pi_{\Phi_2} \widetilde{\Pi}_V u - u \in H_{\Phi_2}(V)
\endeq
and with $\Phi_2$ in place of $\Phi_1$. Here in the second equality in (\ref{eq4.92}) we have also used (\ref{eq4.5}). We obtain, using also (\ref{eq4.3}),
\begin{multline}
\label{eq4.10}
\abs{\left((\widetilde{\Pi}_V-1)u, \Pi_{\Phi_2}\left((\widetilde{\Pi}_V -1)u\right)\right)_{L^2_{\Phi}(V)}} \leq
{\cal O}(h^{\infty})\norm{u}_{H_{\Phi}(V)} \norm{\Pi_{\Phi_2} \widetilde{\Pi}_V u - u}_{H_{\Phi_2}(V)} \\
\leq {\cal O}(h^{\infty})\norm{u}^2_{H_{\Phi_1}(V)}.
\end{multline}
Next, we write 
\begeq
\label{eq4.11}
\Pi_{\Phi_2}\left((\widetilde{\Pi}_V -1)u\right) = (\widetilde{\Pi}_V -1)u + Ru,
\endeq
where $R = \Pi_{\Phi_2} \widetilde{\Pi}_V - \widetilde{\Pi}_V$. We get, combining (\ref{eq4.10}), (\ref{eq4.11}), and using that
$\Phi_j \leq \Phi$, for $j=1,2$, together with (\ref{eq3.3}),
\begin{multline}
\label{eq4.12}
\norm{(\widetilde{\Pi}_V -1)u}^2_{L^2_{\Phi}(V)} \leq {\cal O}(h^{\infty})\norm{u}^2_{H_{\Phi_1}(V)} +
{\cal O}(1) \norm{u}_{H_{\Phi}(V)} \norm{Ru}_{L^2_{\Phi}(V)} \\
\leq {\cal O}(h^{\infty})\norm{u}^2_{H_{\Phi_1}(V)} +
{\cal O}(1) \norm{u}_{H_{\Phi}(V)} \norm{Ru}_{L^2_{\Phi_2}(V)} \\
\leq {\cal O}(h^{\infty})\norm{u}^2_{H_{\Phi_1}(V)}
+ {\cal O}(h^{\infty}) \norm{u}_{H_{\Phi}(V)} \norm{u}_{H_{\Phi_1}(V)} \leq {\cal O}(h^{\infty})\norm{u}^2_{H_{\Phi_1}(V)}.
\end{multline}
Here in the penultimate inequality we have also used Proposition \ref{prop_orth_proj}. We obtain therefore that
\begeq
\label{eq4.13}
\norm{(\widetilde{\Pi}_V -1)u}_{L^2_{\Phi}(V)} \leq {\cal O}(h^{\infty})\norm{u}_{H_{\Phi_1}(V)},\quad u\in H_{\Phi_1}(V),
\endeq
where $\Phi_1 \in C^{\infty}(\Omega;\real)$ is of the form (\ref{eq4.1}). The bound (\ref{eq4.13}) is completely analogous to the estimate (5.5) in~\cite[Section 5]{DelHiSj}, and we may therefore conclude the proof of Theorem \ref{Theorem1} by repeating the arguments of~\cite[Section 5]{DelHiSj}, which are based on the $\overline{\partial}$--techniques, exactly as they stand. Letting $U \Subset W \Subset V$ be an open neighborhood of $x_0$ with $C^{\infty}$--boundary, we get therefore,
\begeq
\label{eq4.14}
\norm{(\widetilde{\Pi}_V -1)u}_{L^2_{\Phi}(U)} \leq {\cal O}(h^{\infty}) \norm{u}_{H_{\Phi}(V)},\quad u\in H_{\Phi}(V).
\endeq
The proof of Theorem \ref{Theorem1} is complete.

\section{From asymptotic local to global Bergman kernels}
\label{secLG}
\setcounter{equation}{0}
The purpose of this section is to establish a link between the operator $\widetilde{\Pi}_V$ in (\ref{eq1.3}), enjoying the local approximate reproducing property (\ref{eq1.4}), and the orthogonal projection
\begeq
\label{LG1}
\Pi: L^2(\Omega, e^{-2\Phi/h}L(dx)) \rightarrow H_{\Phi}(\Omega).
\endeq
When doing so, we shall follow~\cite{BBSj} closely, where the Bergman projection was considered in the context of high powers of a holomorphic line bundle over a complex compact manifold. The discussion below is therefore essentially well known, and is given here mainly for the completeness and convenience of the reader. See also~\cite{E1},~\cite{E2}.

\medskip
\noindent
We shall assume in what follows that the open set $\Omega \subset \comp^n$ is pseudoconvex. It will also be convenient for us to choose a local polarization $\Psi$ of $\Phi\in C^{\infty}(\Omega)$ satisfying (\ref{eq2.2}), (\ref{eq2.3}), such that the Hermitian
property
\begeq
\label{LG2}
\Psi(x,y) = \overline{\Psi(\overline{y},\overline{x})},\quad (x,y)\in {\rm neigh}((x_0,\overline{x_0}), \comp^{2n}_{x,y})
\endeq
holds. Indeed, if $\Psi(x,y)$ satisfies (\ref{eq2.2}), (\ref{eq2.3}), then so does $\overline{\Psi(\overline{y},\overline{x})}$, and replacing $\Psi(x,y)$ by $\displaystyle (\Psi(x,y)+\overline{\Psi(\overline{y},\overline{x})})/2$, we obtain (\ref{LG2}).

\bigskip
\noindent
Our starting point is the following well known result, allowing us to pass to pointwise estimates from the weighted $L^2$ estimates in Theorem \ref{Theorem1}.
\begin{prop}
\label{Horm_estimate}
Let $V_1\Subset V_2\Subset\Omega$ be open. There exists $C>0$ such that for all $f\in L^2_{\Phi}(V_2)$ satisfying $h\overline{\partial}f\in L^{\infty}(V_2)$ and all $h>0$ small enough, we have
\begeq
\label{LG3}
\abs{f(x)} \leq C\left(\sup_{y\in V_2} \abs{h\overline{\partial}f(y)}e^{-\Phi(y)/h} +h^{-n}\norm{f}_{L^2_\Phi(V_2)}\right)e^{\Phi(x)/h}, \quad x\in V_1.
\endeq
\end{prop}
\begin{proof}
Let $h_0 > 0$ be such that $B(x,h_0) = \{y\in \comp^n, \abs{y-x}<h_0\} \subset V_2$, for all $x\in V_1$. An application of~\cite[Lemma 15.1.8]{H_book} gives for all $h\in (0,h_0]$,
\begeq
\label{LG4}
\abs{f(x)} \leq C \left(\sup_{B(x,h)} \abs{h\overline{\partial}f(y)}+ h^{-n} \norm{f}_{L^2(B(x,h))}\right),\quad x\in V_1.
\endeq
Using that
$$
e^{\Phi(y)/h} \leq {\cal O}(1) e^{\Phi(x)/h},\quad y\in B(x,h),
$$
for $x\in V_1$, we obtain (\ref{LG3}).
\end{proof}

\medskip
\noindent
We shall apply Proposition \ref{Horm_estimate} to a function of the form
\begeq
\label{LG5}
f = \widetilde{\Pi}_V u - u,\quad u\in H_{\Phi}(V),
\endeq
satisfying
\begeq
\label{LG6}
\norm{f}_{L^2_{\Phi}(U)} \leq {\cal O}(h^\infty) \norm{u}_{H_{\Phi}(V)},
\endeq
in view of (\ref{eq1.4}). As for the control of $h\overline{\partial} f = h\overline{\partial} \widetilde{\Pi}_V u$ in $L^{\infty}(U)$, we have in view of (\ref{eq2.3}), (\ref{eq2.70}), (\ref{eq2.71}), for $N = 1,2,\ldots\,$,
\begin{multline}
\label{LG7}
\abs{h\overline{\partial} f(x)} \leq {\cal O}_N(1) h^{-n} \int_V e^{\frac{2}{h}{\rm Re}\, \Psi(x,\overline{y})} \abs{x-y}^N \abs{u(y)} e^{-\frac{2}{h}\Phi(y)}\,L(dy)\\
\leq {\cal O}_N(1)e^{\Phi(x)/h} h^{-n} \int_V e^{-\abs{x-y}^2/Ch} \abs{x-y}^N \abs{u(y)} e^{-\Phi(y)/h}\, L(dy) \\
\leq {\cal O}(h^{\frac{N}{2}-n}) e^{\Phi(x)/h}\norm{u}_{H_{\Phi}(V)},\quad x\in U.
\end{multline}
Here we have also used Proposition \ref{prop1} and the Cauchy-Schwarz inequality. Letting $\widetilde{U}\Subset U$ be an open neighborhood of $x_0$, and combining Proposition \ref{Horm_estimate} with (\ref{LG6}), (\ref{LG7}), we get
\begeq
\label{LG8}
\abs{\widetilde{\Pi}_Vu(x) - u(x)} \leq {\cal O}(h^{\infty})\, e^{\Phi(x)/h}\norm{u}_{H_{\Phi}(V)},\quad x\in \widetilde{U}.
\endeq

\medskip
\noindent
We obtain therefore the following approximate local reproducing property,
\begeq
\label{LG9}
u(x) = \frac{1}{h^n} \int_{V} e^{\frac{2}{h}\Psi(x,\overline{y})} a(x,\overline{y};h)  u(y) e^{-\frac{2}{h}\Phi(y)}\, L(dy) + {\cal O}(h^\infty) e^{\frac{\Phi(x)}{h}}\norm{u}_{H_{\Phi}(V)}, \quad x\in \widetilde{U},
\endeq
valid for all $u\in H_{\Phi}(V)$. In particular, we can apply (\ref{LG9}) to $u\in H_{\Phi}(\Omega)$, and when doing so, let us recall that
\begeq
\label{LG10}
\abs{u(x)} \leq {\cal O}(1) h^{-n} e^{\frac{\Phi(x)}{h}}\norm{u}_{H_{\Phi}(\Omega)}, \quad x\in V,
\endeq
in view of~\cite[Proposition 2.3]{DelHiSj}. Let $\chi \in C^{\infty}_0(V;[0,1])$ be such that $\chi = 1$ near $\overline{U}$. Using (\ref{LG10}) and Proposition \ref{prop1}, we see that
\begin{multline}
\label{LG11}
\abs{\frac{1}{h^n} \int_{V} e^{\frac{2}{h}\Psi(x,\overline{y})} (1-\chi(y)) a(x,\overline{y};h)  u(y) e^{-\frac{2}{h}\Phi(y)}\, L(dy)} \\
\leq {\cal O}(1) e^{-\frac{1}{Ch}} e^{\frac{\Phi(x)}{h}}\norm{u}_{H_{\Phi}(\Omega)}, \quad x\in \widetilde{U}.
\end{multline}
We get, combining (\ref{LG9}) and (\ref{LG11}), when $u\in H_{\Phi}(\Omega)$,
\begeq
\label{LG12}
u(y) = \int_{\Omega} \widetilde{K}(y,\overline{z}) \chi(z) u(z) e^{-\frac{2}{h}\Phi(z)}\, L(dz) + {\cal O}(h^\infty) e^{\frac{\Phi(y)}{h}}\norm{u}_{H_{\Phi}(\Omega)}, \quad y\in \widetilde{U},
\endeq
where
\begeq
\label{LG13}
\widetilde{K}(y,\overline{z}) = \frac{1}{h^n} e^{\frac{2}{h}\Psi(y,\overline{z})} a(y,\overline{z};h),\quad (y,z)\in V \times V.
\endeq

\bigskip
\noindent
Next, arguing as in~\cite[Section 5]{RSV},~\cite[Appendix A]{CoHiSj}, we see that the Schwartz kernel of the orthogonal projection in (\ref{LG1}) is of the form $K(x,\overline{y}) e^{-2\Phi(y)/h}$, where $K(x,z)\in {\rm Hol}(\Omega \times \overline{\Omega})$ satisfies
\begeq
\label{LG14}
y\mapsto \overline{K(x,\overline{y})} \in H_{\Phi}(\Omega), \quad x\mapsto K(x,\overline{y}) \in H_{\Phi}(\Omega).
\endeq
Following~\cite{BBSj} and applying (\ref{LG12}) to the function $y\mapsto K(y,\overline{x})\in H_{\Phi}(\Omega)$, we get
\begin{multline}
\label{LG15}
K(y,\overline{x}) = \int_{\Omega} \widetilde{K}(y,\overline{z}) \chi(z) K(z,\overline{x}) e^{-\frac{2}{h}\Phi(z)}\, L(dz) \\ + {\cal O}(h^\infty) e^{\frac{\Phi(y)}{h}}\norm{K(\cdot,\overline{x})}_{H_{\Phi}(\Omega)}, \quad y\in \widetilde{U}.
\end{multline}
Here we have
\begeq
\label{LG16}
\norm{K(\cdot,\overline{x})}_{H_{\Phi}(\Omega)} \leq {\cal O}(1) h^{-n/2} e^{\frac{\Phi(x)}{h}}, \quad x\in \widetilde{U},
\endeq
see~\cite[Chapter 4]{B10}, and combining (\ref{LG15}) and (\ref{LG16}), we infer that
\begeq
\label{LG17}
K(y,\overline{x}) = \int_{\Omega} \widetilde{K}(y,\overline{z}) \chi(z) K(z,\overline{x}) e^{-\frac{2}{h}\Phi(z)}\, L(dz) + {\cal O}(h^\infty) e^{\frac{\Phi(x) + \Phi(y)}{h}}, \quad x, y\in \widetilde{U}.
\endeq
Taking the complex conjugates in (\ref{LG17}) and using the Hermitian property $K(x,\overline{y}) = \overline{K(y,\overline{x})}$, we get
\begeq
\label{LG18}
K(x,\overline{y}) = \int_{\Omega} K(x,\overline{z}) \chi(z) \widehat{K}(z,\overline{y}) e^{-\frac{2}{h}\Phi(z)}\, L(dz) + {\cal O}(h^\infty)e^{\frac{\Phi(x) + \Phi(y)}{h}}, \quad x, y\in \widetilde{U},
\endeq
where, in view of (\ref{LG13}) and (\ref{LG2}),
\begeq
\label{LG19}
\widehat{K}(z,\overline{y}) = \overline{\widetilde{K}(y,\overline{z})} = \frac{1}{h^n} e^{\frac{2}{h}\Psi(z,\overline{y})} b(z,\overline{y};h),\quad b(z,\overline{y};h) = \overline{a(y,\overline{z};h)}.
\endeq
Writing
\begeq
\label{LG20}
\Pi u(x) = \int_{\Omega} K(x,\overline{y})u(y) e^{-2\Phi(y)/h}\, L(dy),\quad u\in L^2(\Omega,e^{-2\Phi/h}L(dx)),
\endeq
we may express (\ref{LG18}) as follows,
\begeq
\label{LG21}
K(x,\overline{y}) = \Pi\left(\widehat{K}(\cdot,\overline{y})\chi\right)(x) + {\cal O}(h^\infty) e^{\frac{\Phi(x) + \Phi(y)}{h}}, \quad x, y\in \widetilde{U}.
\endeq
Here we would like to show that $\Pi\left(\widehat{K}(\cdot,\overline{y})\chi\right)(x)$ is close to $\widehat{K}(x,\overline{y})\chi(x) = \widehat{K}(x,\overline{y})$ for $x\in \widetilde{U}$, and to this end we follow an argument in~\cite{BBSj}, see also Proposition \ref{prop_orth_proj}. The function
\begeq
\label{LG22}
\Omega \ni x \mapsto u_y(x) = \widehat{K}(x,\overline{y})\chi(x) - \Pi\left(\widehat{K}(\cdot,\overline{y})\chi\right)(x)
\endeq
is the solution of the $\overline{\partial}$--equation
\begeq
\label{LG23}
\overline{\partial} u_y = \overline{\partial}\left(\widehat{K}(\cdot,\overline{y})\chi\right) = \widehat{K}(\cdot,\overline{y})\overline{\partial}\chi+\chi\overline{\partial}\widehat{K}(\cdot,\overline{y}),
\endeq
in $\Omega$ of the minimal $L^2(\Omega,e^{-2\Phi/h}L(dx))$ norm, and therefore, by H\"ormander's $L^2$--estimates for the $\overline{\partial}$ operator, see~\cite[Proposition 4.2.5]{Horm_Conv}), we get for any $y\in \widetilde{U}$,
\begin{multline}
\label{LG24}
\int_{\Omega} \abs{u_y(x)}^2 e^{-2\Phi(x)/h}\, L(dx) \leq {\cal O}(h) \int_{\Omega} \frac{1}{c(x)}\abs{\overline{\partial}_x\left(\widehat{K}(x,\overline{y})\chi(x)\right)}^2 e^{-2\Phi(x)/h}\, L(dx) \\
\leq {\cal O}(h) \left(\int_{\Omega} \abs{\nabla \chi(x)}^2 \abs{\widehat{K}(x,\overline{y})}^2 e^{-2\Phi(x)/h}\, L(dx)+\int_{\Omega}\chi(x)\abs{\overline{\partial}_x\widehat{K}(x,\overline{y})}^2 e^{-2\Phi(x)/h}\, L(dx)\right)
\end{multline}
Here we get, in view of (\ref{LG19}) and Proposition \ref{prop1},
\begin{multline}
\label{LG25}
\int_{\Omega} \abs{\nabla \chi(x)}^2 \abs{\widehat{K}(x,\overline{y})}^2 e^{-2\Phi(x)/h}\, L(dx) \leq
{\cal O}(1) \int_{V\backslash U} \abs{\widehat{K}(x,\overline{y})}^2 e^{-2\Phi(x)/h}\, L(dx)\\
={\cal O}(1)\, e^{2\Phi(y)/h}e^{-1/Ch},\quad y\in \widetilde{U}.
\end{multline}
Due to the almost holomorphy of $\Psi$ and the symbol $b(x,\widetilde{y};h)$ in (\ref{LG19}) near the anti-diagonal,
\begeq
\label{LG26}
\abs{\overline{\partial}_x\widehat{K}(x,\overline{y})} \leq \frac{{\cal O}_N(1)}{h^{n+1}} e^{\frac{2}{h}{\rm Re}\, \Psi(x,\overline{y})}\abs{x-y}^N, \quad N=1,2,\ldots,
\endeq
and therefore, by another application of Proposition \ref{prop1},
\begeq
\label{LG27}
\int_{\Omega}\chi(x)\abs{\overline{\partial}_x\widehat{K}(x,\overline{y})}^2 e^{-2\Phi(x)/h}\, L(dx) \leq {\cal O}(h^{\infty})\, e^{2\Phi(y)/h},\quad y\in \widetilde{U}.
\endeq
Combining (\ref{LG24}), (\ref{LG25}), and (\ref{LG27}), we get
\begeq
\label{LG28}
\norm{u_y}_{L^2_\Phi(\Omega)} \leq {\cal O}(h^\infty)e^{\Phi(y)/h},\quad y\in \widetilde{U}.
\endeq

\bigskip
\noindent
It finally remains for us to pass from the weighted $L^2$--bound (\ref{LG28}) on $u_y$ to a pointwise estimate. To this end, we shall apply Proposition \ref{Horm_estimate} to $u_y$, with $V_2 = U$, $V_1 = \widetilde{U}$. Recalling (\ref{LG22}) and using that $\chi = 1$ on $U$, we see that we only have to estimate $h\overline{\partial}_zu_y(z) = h\overline{\partial}_z \widehat{K}(z,\overline{y})$ for $z\in U$, $y\in \widetilde{U}$. Using (\ref{LG26}) and Proposition \ref{prop1}, we obtain that
\begeq
\label{LG29}
\abs{h\overline{\partial}u_y(z)} \leq {\cal O}(h^{\infty}) e^{(\Phi(z) + \Phi(y))/h},\quad y\in \widetilde{U},\,\,z\in U.
\endeq
Combining (\ref{LG3}), (\ref{LG28}), and (\ref{LG29}), we get
\begeq
\label{LG30}
\abs{u_y(x)} \le {\cal O}(h^\infty)\,e^{(\Phi(x)+\Phi(y))/h},\quad x,y\in \widetilde{U},
\endeq
and therefore, using (\ref{LG21}), (\ref{LG22}), and (\ref{LG30}), we infer that
\begeq
\label{LG31}
K(x,\overline{y}) = \widehat{K}(x,\overline{y}) + {\cal O}(h^\infty) e^{(\Phi(x)+\Phi(y))/h},\quad x, y\in \widetilde{U}.
\endeq
Recalling also (\ref{LG19}), we obtain
\begeq
\label{LG32}
K(x,\overline{y}) = \overline{\widetilde{K}(y,\overline{x})} + {\cal O}(h^\infty) e^{(\Phi(x)+\Phi(y))/h},\quad x, y\in \widetilde{U},
\endeq
and taking the complex conjugates and using the Hermitian symmetry of $K$, we get
\begeq
\label{LG33}
K(y,\overline{x}) = \widetilde{K}(y,\overline{x}) + {\cal O}(h^\infty) e^{(\Phi(x)+\Phi(y))/h},\quad x, y\in \widetilde{U}.
\endeq
Switching the variables $x$ and $y$ in (\ref{LG33}), we may therefore summarize the discussion in this section in the following well known result, see~\cite{Z98},~\cite{Catlin},~\cite{BBSj},~\cite{Charles},~\cite{DLM},~\cite{E2}.
\begin{prop}
Let $\Omega \subset \comp^n$ be open pseudoconvex, let $\Phi \in C^{\infty}(\Omega)$ be strictly plurisubharmonic so that {\rm (\ref{eq2.1})} holds, and let $K(x,\overline{y})e^{-2\Phi(y)/h}$ be the Schwartz kernel of the orthogonal projection {\rm (\ref{LG1})}. Let $x_0 \in \Omega$ and let $\widetilde{U} \Subset U \Subset V \Subset \Omega$ be small open neighborhoods of $x_0$, where $U$, $V$ are as in Theorem {\rm \ref{Theorem1}}. We have
\begeq
\label{LG34}
e^{-\Phi(x)/h}\left(K(x,\overline{y}) - \frac{1}{h^n} e^{\frac{2}{h}\Psi(x,\overline{y})}a(x,\overline{y};h)\right)e^{-\Phi(y))/h}  = {\cal O}(h^\infty),
\endeq
uniformly for $x, y\in \widetilde{U}$. Here $\Psi\in C^{\infty}({\rm neigh}((x_0,\overline{x_0}),\comp^{2n}))$ is a polarization of $\Phi$ and the classical symbol $a\in S^0_{{\rm cl}}({\rm neigh}((x_0,\overline{x_0}),\comp^{2n}))$ has been introduced in {\rm (\ref{eq1.2})}.
\end{prop}

\end{document}